\documentclass[11pt,a4paper]{article}
\usepackage[utf8]{inputenc}
\usepackage{fullpage}
\usepackage{color,amsmath}
\usepackage{amsthm}
\usepackage{hyperref}
\usepackage{amssymb}
\usepackage{amsbsy}
\usepackage{cite}
\usepackage{mathrsfs}
\usepackage{amstext}
\usepackage{enumerate}
\usepackage{graphics}
\usepackage{graphicx}
\usepackage{enumitem}
\usepackage[capitalise]{cleveref}

\newcommand{\NP}{{NP}}

\newtheorem{theorem}{Theorem}[section]
\newtheorem{proposition}[theorem]{Proposition}
\newtheorem{lemma}[theorem]{Lemma}
\newtheorem{corollary}[theorem]{Corollary}
\theoremstyle{definition}
\newtheorem{remark}[theorem]{Remark}

\title{Bisimplicial separators}
\author{Martin Milani\v{c}\thanks{FAMNIT and IAM, University of Primorska, Koper, Slovenia. Email: \texttt{martin.milanic@upr.si}. Partially supported by the Slovenian Research Agency (I0-0035, research program P1-0285 and research projects N1-0102, N1-0160, J1-3001, J1-3002, J1-3003, J1-4008, and J1-4084).}
\and Irena Penev\thanks{Computer Science Institute of Charles University (I\'{U}UK), Prague, Czech Republic. Email: \texttt{ipenev@iuuk.mff.cuni.cz}.} \and Nevena Piva\v{c}\thanks{FAMNIT and IAM, University of Primorska, Koper, Slovenia. Email: \texttt{nevena.pivac@iam.upr.si}. Partially supported by the Slovenian Research Agency (research program P1-0285, research projects N1-0102 and J1-4008, and a Young Researchers Grant).} 
\and Kristina Vu\v{s}kovi\'c\thanks{School of Computing, University of Leeds, Leeds, UK. Email: \texttt{k.vuskovic@leeds.ac.uk}. Partially supported by
EPSRC grant EP/V002813/1.}}
\begin{document}
\maketitle

\begin{abstract}
A \emph{minimal separator} of a graph $G$ is a set $S \subseteq V(G)$ such that there exist vertices $a,b \in V(G) \setminus S$ with the property that $S$ separates $a$ from $b$ in $G$, but no proper subset of $S$ does.
For an integer $k\ge 0$, we say that a minimal separator is \emph{$k$-simplicial} if it can be covered by $k$ cliques and denote by $\mathcal{G}_k$ the class of all graphs in which each minimal separator is $k$-simplicial.
We show that for each $k \geq 0$, the class $\mathcal{G}_k$ is closed under induced minors, and we use this to show that the \textsc{Maximum Weight Stable Set} problem can be solved in polynomial time for $\mathcal{G}_k$.
We also give a complete list of minimal forbidden induced minors for $\mathcal{G}_2$.
Next, we show that, for $k \geq 1$, every nonnull graph in $\mathcal{G}_k$ has a $k$-simplicial vertex, i.e., a vertex whose neighborhood is a union of $k$ cliques; we deduce that the \textsc{Maximum Weight Clique} problem can be solved in polynomial time for graphs in $\mathcal{G}_2$.
Further, we show that, for $k \geq 3$, it is \NP-hard to recognize graphs in $\mathcal{G}_k$; the time complexity of recognizing graphs in $\mathcal{G}_2$ is unknown. 
We also show that the \textsc{Maximum Clique} problem is \NP-hard for graphs in $\mathcal{G}_3$. Finally, we prove a decomposition theorem for diamond-free graphs in $\mathcal{G}_2$ (where the \emph{diamond} is the graph obtained from $K_4$ by deleting one edge), and we use this theorem to obtain polynomial-time algorithms for the \textsc{Vertex Coloring} and recognition problems for diamond-free graphs in $\mathcal{G}_2$, and improved running times for the \textsc{Maximum Weight Clique} and \textsc{Maximum Weight Stable Set} problems for this class of graphs.
\end{abstract}


\section{Introduction}

All graphs in this paper are finite, simple, and undirected. Our graphs may possibly be null.
For a graph $G$ and nonadjacent vertices $a,b \in V(G)$,
\begin{itemize}
\item an \emph{$(a,b)$-separator} of $G$ is a set $S \subseteq V(G) \setminus \{a,b\}$ such that $a$ and $b$ belong to distinct components of $G \setminus S$;
\item a \emph{minimal $(a,b)$-separator} of $G$ is an $(a,b)$-separator $S$ of $G$ such that no proper subset of $S$ is an $(a,b)$-separator of $G$.
\end{itemize}

For a graph $G$, a set $S \subseteq V(G)$ is a \emph{separator} (resp.\ \emph{minimal separator}) of $G$ if there exist distinct, nonadjacent vertices $a,b \in V(G) \setminus S$ such that $S$ is an $(a,b)$-separator (resp.\ minimal $(a,b)$-separator) of $G$. Note that it is possible that $S$ is a minimal separator of a graph $G$, even though some $S' \subsetneqq S$ is also a separator of $G$. Indeed, there may be a pair $a,b$ of nonadjacent vertices such that $S$ is a minimal $(a,b)$-separator of $G$, as well as some other pair $a',b'$ of nonadjacent vertices such that some $S' \subsetneqq S$ is an $(a',b')$-separator of $G$.

A graph is \emph{chordal} if it contains no induced cycles of length greater than three.
Minimal separators have been studied since at least the 1960s, when chordal graphs were characterized as precisely those graphs in which all minimal separators are cliques~\cite{D61}.
Minimal separators were subsequently studied in~\cite{moplex} in the context of moplexes, have played an important role in sparse matrix computations via minimal triangulations (for a survey, see~\cite{MinTriangulations}), and have also had numerous algorithmic applications (see, e.g.,~\cite{BackerATFree, SafeSepTW, TWFillIn, tarjan1985decomposition}).
This paper is a contribution to the study of minimal separators.

For a class $\mathcal{C}$ of graphs, we denote by $\mathcal{G}_{\mathcal{C}}$ the class of all graphs $G$ such that every minimal separator of $G$ induces a graph from $\mathcal{C}$.
Since complete graphs have no separators, we see for all classes $\mathcal{C}$, the class $\mathcal{G}_{\mathcal{C}}$ contains all complete graphs (including the null graph).
For a nonnegative integer $k$, we denote by $\mathcal{G}_k$ the class of all graphs $G$ that have the property that every minimal separator of $G$ is a union of $k$ (possibly empty) cliques.\footnote{Clearly, given a graph $G$ and a set $S\subseteq V(G)$, the set $S$ is a union of $k$ cliques if and only if $S$ is a union of $k$ pairwise disjoint cliques.}
Obviously, $\mathcal{G}_0 \subseteq \mathcal{G}_1 \subseteq \mathcal{G}_2 \subseteq \dots$, and all these inclusions are proper, as verified by the class of complete bipartite graphs with exactly two vertices in one part.
Note that $\mathcal{G}_0$ is the class of all disjoint unions of (arbitrarily many) complete graphs, and (by~\cite{D61}) $\mathcal{G}_1$ is the class of all chordal graphs.
For a nonnegative integer $k$, we denote by $\mathcal{C}_k$ the class of all graphs whose vertex sets can be partitioned into $k$ (possibly empty) cliques; clearly, $\mathcal{G}_k = \mathcal{G}_{\mathcal{C}_k}$.

In this paper, we prove a number of results about classes of the form $\mathcal{G}_{\mathcal{C}}$, where $\mathcal{C}$ is a hereditary class.
We place particular emphasis on the class $\mathcal{G}_2$. By the above, $\mathcal{G}_2$ contains all chordal graphs.
Moreover, it is easy to see that all circular-arc graphs (that is, intersection graphs of arcs on a circle) belong to $\mathcal{G}_2$.

In \cref{sec:basic}, we prove some basic properties of the class $\mathcal{G}_{\mathcal{C}}$, when $\mathcal{C}$ satisfies various hypotheses.

In \cref{sec:ForbIndMinor}, we focus on hereditary graph classes that are closed under edge addition.
We show that for any such class $\mathcal{C}$, both the class $\mathcal{C}$ and the corresponding class $\mathcal{G}_\mathcal{C}$ are closed under induced minors, and we characterize the class $\mathcal{G}_\mathcal{C}$ in terms of forbidden induced minors. 
In the case when $\mathcal{C}$ is also closed under the addition of universal vertices, the class of \emph{minimal} forbidden induced minors for $\mathcal{G}_\mathcal{C}$ is described precisely.
As a consequence of these results, we obtain that the classes $\mathcal{G}_k$ ($k \geq 0$) are closed under induced minors and give a complete list of minimal forbidden induced minors for the class $\mathcal{G}_2$. 
Combining these results with some results from the literature~\cite{TreeDecompBoundAlpha, MR3647815}, we show that for every integer $k \geq 0$, the \textsc{Maximum Weight Stable Set Problem} can be solved in polynomial time for graphs in $\mathcal{G}_k$, and we further show that all 1-perfectly-orientable graphs belong to $\mathcal{G}_2$.

In \cref{sec:k-simplicial}, we show that for every $k \geq 1$, every nonnull graph in $\mathcal{G}_k$ has a $k$-simplicial vertex (i.e., a vertex whose neighborhood is a union of $k$ cliques).
We do this by showing that every LexBFS ordering of a graph in $\mathcal{G}_k$ is a $k$-simplicial elimination ordering.\footnote{We postpone the statements of precise definitions to~\cref{sec:k-simplicial}.} 
This generalizes a result by Rose, Tarjan, and Lueker on chordal graphs~\cite{rose1976algorithmic}, which corresponds to the case $k = 1$
Using~\cite{ElimGraphs} we deduce that the \textsc{Maximum Weight Clique} problem can be solved in polynomial time for graphs in $\mathcal{G}_2$.

In \cref{sec:recGGK}, we show that for each $k\ge 3$, it is \NP-hard to recognize graphs in $\mathcal{G}_k$; the time complexity of recognizing graphs in $\mathcal{G}_2$ is unknown. We further show that the \textsc{Maximum Clique} problem is \NP-hard for $\mathcal{G}_3$ (and consequently for $\mathcal{G}_k$ whenever $k \geq 3$).
Note that, since \textsc{Vertex Coloring} is \NP-hard for circular-arc graphs~\cite{CircArcColNPComp}, which form a subclass of $\mathcal{G}_2$, the problem is also \NP-hard for $\mathcal{G}_k$, whenever $k \geq 2$.

The \emph{diamond} is the four-vertex graph obtained from the complete graph $K_4$ by deleting one edge.
In \cref{sec:diamond}, we prove a decomposition theorem for diamond-free graphs in $\mathcal{G}_2$, and use this theorem to obtain polynomial-time algorithms for the \textsc{Vertex Coloring} and recognition problems for diamond-free graphs in $\mathcal{G}_2$, and improved running times for the \textsc{Maximum Weight Clique} and \textsc{Maximum Weight Stable Set} problems in this class of graphs.

\cref{table1} summarizes our algorithmic and complexity results.
Since $\mathcal{G}_0$ is the class of all disjoint unions of complete graphs, and $\mathcal{G}_1$ is the class of all chordal graphs, all problems from the table below can be solved in linear time for $\mathcal{G}_0$ and $\mathcal{G}_1$ (see~\cite{rose1976algorithmic,MR0392683}).

\begin{table}
\begin{center}
\begin{tabular}{|l || c | c | c |}
 \hline
  & diamond-free & & \\
  & graphs in $\mathcal{G}_2$ & $\mathcal{G}_2$ & $\mathcal{G}_k$ ($k \geq 3$) \\
[0.5ex]
 \hline\hline
 recognition & $\mathcal{O}(n^{\omega}\log n)$ & ? & \NP-hard \\
 \hline
 \textsc{Maximum Weight Clique} & $\mathcal{O}(n^{\omega}\log n)$ & $\mathcal{O}(n^{3+o(1)})$ & \NP-hard \\
 \hline
 \textsc{Maximum Weight Stable Set} & $\mathcal{O}(n^2(n+m))$ & $\mathcal{O}(n^6)$ & $\mathcal{O}(n^{2k+2})$ \\
 \hline
 \textsc{Vertex Coloring} & $\mathcal{O}(n^{\omega}\log n)$ & \NP-hard & \NP-hard \\
 \hline
\end{tabular}
\caption{Summary of our algorithmic and complexity results.
The number of vertices and edges of the input graph are denoted by $n$ and $m$, respectively, and $\omega<2.3728596$ denotes the matrix multiplication exponent (see~\cite{MR4262465}).
}\label{table1}
\end{center}
\end{table}

\subsection{Terminology and notation} \label{subsec:def}

The vertex set and the edge set of a graph $G$ are denoted by $V(G)$ and $E(G)$, respectively.
The complement of $G$ is denoted by $\overline{G}$.

A \emph{clique} in a graph $G$ is a (possibly empty) set of pairwise adjacent vertices, and a \emph{stable set} in $G$ is a (possibly empty) set of pairwise nonadjacent vertices.
Given a graph $G$ and a vertex $x \in V(G)$, we denote by $N_G(x)$ the set of all neighbors of $x$ in $G$, and we set $N_G[x] := \{x\} \cup N_G(x)$.
For a nonnegative integer $k$, a vertex $x \in V(G)$ is \emph{$k$-simplicial} in $G$ if $N_G(x)$ is the union of $k$ cliques.
A $1$-simplicial vertex is also called \emph{simplicial}, and a $2$-simplicial vertex is also called \emph{bisimplicial}.
Analogously, for a graph $G$ and a set $X \subseteq V(G)$, we say that $X$ is \emph{$k$-simplicial} in $G$ if it is a union of $k$ cliques and
\emph{bisimplicial} if it is $2$-simplicial.
In particular, a graph $G$ belongs to the class $\mathcal{G}_k$ if and only if every minimal separator of $G$ is $k$-simplicial, and to $\mathcal{G}_2$ if and only if every minimal separator of $G$ is bisimplicial.

We say that a graph $H$ is an \emph{induced subgraph} of a graph $G$ if $H$ can be obtained from $G$ via a (possibly null) sequence of vertex deletions.
Given graphs $H$ and $G$, we say that $G$ is \emph{$H$-free} if no induced subgraph of $G$ is isomorphic to $H$.
For a graph $G$ and a set $X \subseteq V(G)$, we denote by $G[X]$ the subgraph of $G$ induced by $X$; if $X = \{x_1,\dots,x_t\}$, we sometimes write $G[x_1,\dots,x_t]$ instead of $G[X]$.
Furthermore, $G \setminus X$ is the graph obtained from $G$ by deleting all vertices in $X$, i.e., $G \setminus X := G[V(G) \setminus X]$; if $X = \{x\}$, we sometimes write $G \setminus x$ instead of $G \setminus X$.
A \emph{path} in $G$ is a nonempty sequence $p_1,\ldots,p_k$ of pairwise distinct vertices of $G$ such that $p_i$ and $p_{i+1}$ are adjacent for all $i\in \{1,\ldots, k-1\}$.

For a nonnegative integer $n$, we denote by $K_n$, $P_n$, and $C_n$, respectively, the complete graph, path graph, and cycle graph on $n$ vertices.
For two nonnegative integers $p$ and $q$, we denote by $K_{p,q}$ the complete bipartite graph with parts of size $p$ and $q$.
A class $\mathcal{C}$ of graphs is \emph{hereditary} if for all $G \in \mathcal{C}$, all (isomorphic copies of) induced subgraphs of $G$ belong to $\mathcal{C}$.

Given a graph $G$ and an edge $e= xy\in E(G)$, \emph{subdividing} the edge $e$ means replacing $e$ with a path of length two, that is, removing the edge $e$ and adding a new vertex $z$ adjacent to precisely $x$ and $y$.
A \emph{subdivision} of a graph $G$ is any graph obtained from $G$ by a (possibly null) sequence of edge subdivisions.
Given a graph $G$ and an edge $e\in E(G)$, \emph{contracting} the edge $e = xy$ means replacing the vertices $x$ and $y$ in $G$ with a new vertex $v^{xy}$ adjacent to every vertex that is adjacent in $G$ to $x$ or $y$.
We denote by $G/e$ the graph obtained from $G$ by contracting $e$.
We say that a graph $H$ is an \emph{induced minor} of a graph $G$ if $H$ can be obtained from $G$ via a (possibly null) sequence of vertex deletions and edge contractions.
Given graphs $H$ and $G$, we say that $G$ is \emph{$H$-induced-minor-free} if no induced minor of $G$ is isomorphic to $H$.
For a family $\mathcal{H}$ of graphs, we say that a graph $G$ is \emph{$\mathcal{H}$-induced-minor-free} if $G$ is $H$-induced-minor-free for all graphs $H \in \mathcal{H}$.

Given two graphs $G$ and $H$, an \emph{induced minor model} of $H$ in $G$ is a family $\{X_v\}_{v \in V(H)}$ of nonempty, pairwise disjoint subsets of $V(G)$, each inducing a connected subgraph of $G$, and having the property that for all distinct $u,v \in V(H)$, there is an edge between $X_u$ and $X_v$ in $G$ if and only if $uv \in E(H)$.
Note that $H$ is an induced minor of $G$ if and only if $G$ admits an induced minor model of $H$.

Given a graph $G$, a vertex $x \in V(G)$, and a set $Y \subseteq V(G) \setminus \{x\}$, we say that $x$ is \emph{complete} (resp.\ \emph{anticomplete}) to $Y$ in $G$ if $x$ is adjacent (resp.\ nonadjacent) to all vertices in $Y$. Given disjoint sets $A,B \subseteq V(G)$, we say that $A$ is \emph{complete} (resp.\ \emph{anticomplete}) to $B$ in $G$ if every vertex in $A$ is complete (resp.\ anticomplete) to $B$ in $G$.
A \emph{universal vertex} of a graph $G$ is a vertex $u$ such that $N_G[u] = V(G)$.
\emph{Adding a universal vertex} to a graph $G$ means adding a new vertex $v$ and making it adjacent to all vertices of $G$; note that $v$ is a universal vertex in the resulting graph.

A \emph{cutset} of a graph $G$ is a (possibly empty) set $C \subseteq V(G)$ such that $G \setminus C$ is disconnected. A \emph{minimal cutset} of a graph $G$ is a cutset $C$ of $G$ such that no proper subset of $C$ is a cutset of $G$. A \emph{clique-cutset} is a cutset that is a clique. Note that $\emptyset$ is a clique-cutset of any disconnected graph.

A \emph{cut-partition} of a graph $G$ is a partition $(A,B,C)$ of $V(G)$ such that $A$ and $B$ are nonempty and anticomplete to each other (the set $C$ may possibly be empty).
Note that if $(A,B,C)$ is a cut-partition of $G$, then $C$ is a cutset of $G$; conversely, every cutset of $G$ gives rise to at least one cut-partition of $G$.

If $H_1$ and $H_2$ are graphs on disjoint vertex sets, the \emph{disjoint union} of $H_1$ and $H_2$, denoted by $H_1 \cup H_2$, is the graph with vertex set $V(H_1) \cup V(H_2)$ and edge set $E(H_1) \cup E(H_2)$.
More generally, the operation of \emph{gluing $H_1$ and $H_2$ along a clique} produces a graph obtained from  $H_1 \cup H_2$ by choosing cliques $C_1$ in $H_1$ and $C_2$ in $H_2$ such that $|C_1| = |C_2|$,  fixing a bijection $f$ from $C_1$ to $C_2$, and identifying each vertex $v\in C_1$ with the vertex $f(v)$.

Given a graph $G$ and a vertex weight function $w:V\to\mathbb{Q}_+$, the \textsc{Maximum Weight Clique} problem is the problem of computing a clique $C$ in $G$ with maximum total weight, where the weight of a set $X\subseteq V(G)$ is defined as $\sum_{x\in X}w(x)$.
Similarly, the \textsc{Maximum Weight Stable Set} problem is the problem of computing a stable set $S$ in $G$ with maximum total weight.
In the case of weight functions constantly equal to $1$, we obtain the \textsc{Maximum Clique} and \textsc{Maximum Stable Set} problems, respectively.
A graph $G=(V,E)$ is \emph{$k$-colorable} if $V$ is a union of $k$ stable sets in $G$.
\textsc{Vertex Coloring} is the following problem: given a graph $G$, compute its \emph{chromatic number} $\chi(G)$, that is, the smallest integer $k$ such that $G$ is $k$-colorable.
For a positive integer $k$, \textsc{$k$-Coloring} is the following decision problem: given a graph $G$, determine whether $G$ is $k$-colorable.

\section{Basic properties} \label{sec:basic}

Recall that for any graph class $\mathcal{C}$, we denote by $\mathcal{G}_{\mathcal{C}}$ the class of all graphs $G$ such that every minimal separator of $G$ induces a graph from $\mathcal{C}$.
Note that if $\mathcal{G}_{\mathcal{C}}$ is hereditary, then $\mathcal{C}\subseteq \mathcal{G}_{\mathcal{C}}$.
The following proposition provides some equivalent characterizations of the class $\mathcal{G}_{\mathcal{C}}$ for the case when ${\mathcal{C}}$ is hereditary.

\begin{proposition} \label{prop-sep-cut-equiv} Let $\mathcal{C}$ be a hereditary class of graphs, and let $G$ be a graph. Then the following are equivalent:
\begin{itemize}
\item[(a)] $G \in \mathcal{G}_{\mathcal{C}}$;
\item[(b)] for all induced subgraphs $H$ of $G$, every minimal separator $S$ of $H$ satisfies $H[S] \in \mathcal{C}$;
\item[(c)] for all induced subgraphs $H$ of $G$, every minimal cutset $C$ of $H$ satisfies $H[C] \in \mathcal{C}$.
\end{itemize}
\end{proposition}
\begin{proof}
We prove the result by showing that (a) implies (b), that (b) implies (c), and that (c) implies (a).

First, we assume that (a) holds, and we prove (b). Let $H$ be an induced subgraph of $G$, and suppose that $a,b \in V(H)$ are distinct, nonadjacent vertices, and that $S$ is a minimal $(a,b)$-separator of $H$. Then $S \cup (V(G) \setminus V(H))$ is an $(a,b)$-separator of $G$. Let $S^* \subseteq S \cup (V(G) \setminus V(H))$ be a minimal $(a,b)$-separator of $G$; by (a), $G[S^*] \in \mathcal{C}$. Since $S^*$ is an $(a,b)$-separator of $G$, we have that $S^* \cap V(H)$ is an $(a,b)$-separator of $H$. Moreover, $S^* \cap V(H) \subseteq S$, and so the minimality of $S$ guarantees that $S^* \cap V(H) = S$; consequently, $S \subseteq S^*$. Since $G[S^*] \in \mathcal{C}$, and since $\mathcal{C}$ is hereditary, it follows that $G[S] \in \mathcal{C}$. Clearly, $G[S] = H[S]$, and so $H[S] \in \mathcal{C}$. Thus, (b) holds.

Next, we assume that (b) holds, and we prove (c). Let $H$ be an induced subgraph of $G$, and suppose that $C$ is a minimal cutset of $H$. Let $A$ and $B$ be the vertex sets of two distinct components of $H \setminus C$. The minimality of $C$ guarantees that every vertex in $C$ has a neighbor both in $A$ and in $B$, and this, in turn, guarantees that for all $a \in A$ and $b \in B$, $C$ is a minimal $(a,b)$-separator of $H$. But now (b) implies that $H[C] \in \mathcal{C}$. Thus, (c) holds.

Finally, we assume that (c) holds, and we prove (a). Suppose that $a,b \in V(G)$ are distinct, nonadjacent vertices, and that $S$ is a minimal $(a,b)$-separator of $G$. Let $A$ (resp.\ $B$) be the vertex set of the component of $G \setminus S$ that contains $a$ (resp.\ $b$). Clearly, $A$ and $B$ are disjoint and anticomplete to each other. Furthermore, the minimality of $S$ implies that every vertex in $S$ has a neighbor both in $A$ and in $B$. Set $H := G[A \cup B \cup S]$. Then $(A,B,S)$ is a cut-partition of $H$; furthermore, since $H[A]$ and $H[B]$ are connected, and every vertex of $S$ has a neighbor both in $A$ and in $B$, we see that $S$ is a minimal cutset of $H$. Now (c) guarantees that $H[S] \in \mathcal{C}$; since $H[S] = G[S]$, it follows that $G[S] \in \mathcal{C}$. Thus, (a) holds.
\end{proof}

\begin{corollary} \label{prop-GC-hereditary} 
Let $\mathcal{C}$ be a hereditary class of graphs.
Then $\mathcal{G}_{\mathcal{C}}$ is hereditary.
\end{corollary}

\begin{proof}
This readily follows from \cref{prop-sep-cut-equiv}, and more precisely, from the equivalence of (a) and (b) from \cref{prop-sep-cut-equiv}.
\end{proof}

\begin{theorem} \label{thm-GC-clique-gluing}
Let $\mathcal{C}$ be a hereditary class of graphs that contains all complete graphs.\footnote{However, not all graphs in $\mathcal{C}$ need be complete.}
Then, $\mathcal{G}_{\mathcal{C}}$ is closed under gluing along a clique.
\end{theorem}

\begin{proof}
Let $G$ be a graph that admits a clique-cutset $C$ and let $(A,B,C)$ be an associated cut-partition of $G$.
Assume that $G_A := G[A \cup C]$ and $G_B := G[B \cup C]$ both belong to $\mathcal{G}_{\mathcal{C}}$.
It suffices to show that $G \in \mathcal{G}_{\mathcal{C}}$.

Fix a pair of distinct, nonadjacent vertices $x,y \in V(G)$, and let $S$ be a minimal $(x,y)$-separator of $G$. We must show that $G[S] \in \mathcal{C}$. Since $C$ is a clique, it contains at most one of $x,y$; by symmetry, we may therefore assume that $x \in A$. We now consider two cases: when $y \in A \cup C$, and when $y \in B$.

\medskip

\noindent\textbf{Case~1:} $y \in A \cup C$. Then $S \cap (A \cup C)$ is an $(x,y)$-separator of $G_A$; let $S' \subseteq S \cap (A \cup C)$ be a minimal $(x,y)$-separator of $G_A$. Since $G_A \in \mathcal{G}_{\mathcal{C}}$, we see that $G_A[S'] = G[S']$ belongs to $\mathcal{C}$. If $S' = S$, then we are done. So, assume that $S' \subsetneqq S$. The minimality of $S$ then implies that there is a path $p_1,\dots,p_s$ in $G \setminus S'$, with $p_1 = x$ and $p_s = y$. Since $S'$ is an $(x,y)$-separator of $G_A$, we see that at least one vertex of the path $p_1,\dots,p_s$ belongs to $B$. Let $i$ be the smallest index in $\{1,\dots,s\}$ such that $p_i \in B$, and let $j$ be the largest index in $\{1,\dots,s\}$ such that $p_j \in B$. Since $p_1,p_s \in A \cup C$, we have that $2 \leq i \leq j \leq s-1$. Moreover, since $A$ is anticomplete to $B$, we have that $p_{i-1},p_{j+1} \in C$; since $C$ is a clique, we see that $p_{i-1},p_{j+1}$ are adjacent. But now $p_1,\dots,p_{i-1},p_{j+1},\dots,p_s$ is a path between $x$ and $y$ in $G_A \setminus S'$, contrary to the fact that $S'$ is an $(x,y)$-separator of $G_A$.

\medskip

\noindent\textbf{Case~2:} $y \in B$. Note that, in this case, $C$ is an $(x,y)$-separator of $G$.
\medskip

\noindent
\textbf{Claim.} At least one of $S \cap A$ and $S \cap B$ is empty.

\medskip

\noindent\emph{Proof of the Claim.} Suppose otherwise, i.e., that both $S \cap A$ and $S \cap B$ are nonempty. Set $S_A := S \setminus B$ and $S_B := S \setminus A$. By the minimality of $S$, there is a path $p_1,\dots,p_s$, with $p_1 = x$ and $p_s = y$, in $G \setminus S_A$, and there is a path $q_1,\dots,q_t$, with $q_1 = x$ and $q_t = y$, in $G \setminus S_B$. Since $p_1 \in A$ and $p_s \in B$, and since $A$ is anticomplete to $B$, some internal vertex of the path $p_1,\dots,p_s$ belongs to $C$; let $i$ be the smallest index in $\{2,\dots,s-1\}$ such that $p_i \in C$ (then $p_1,\dots,p_{i-1} \in A$). Similarly, at least one internal vertex of the path $q_1,\dots,q_t$ belongs to $C$; let $j$ be the largest index in $\{2,\dots,t-1\}$ such that $q_j \in C$ (then $q_{j+1},\dots,q_t \in B$). Since $C$ is a clique, we see that $p_i$ and $q_j$ are either equal or adjacent. In the former case, $p_1,\dots,p_i,q_{j+1},\dots,q_t$ is a path between $x$ and $y$ in $G \setminus S$; and in the latter case, $p_1,\dots,p_i,q_j,\dots,q_t$ is a path between $x$ and $y$ in $G \setminus S$. But neither outcome is possible, since $S$ is an $(x,y)$-separator of $G$. This proves the Claim. \hfill$\blacklozenge$

\medskip

By the Claim, and by symmetry, we may assume that $S \cap B = \emptyset$, i.e., $S \subseteq A \cup C$. Let $Y$ be the vertex set of the component of $G[B]$ that contains $y$.

Suppose first that $C \setminus S$ is anticomplete to $Y$. Then $C \cap S$ is an $(x,y)$-separator of $G$, and so the minimality of $S$ guarantees that $S \subseteq C$. Thus, $S$ is a clique; since $\mathcal{C}$ contains all complete graphs, it follows that $G[S] \in \mathcal{C}$, and we are done.

From now on, we assume that $C \setminus S$ is not anticomplete to $Y$. Fix a vertex $c \in C \setminus S$ that has a neighbor in $Y$.
Since $y \in Y$, and $G[Y]$ is connected, the graph $G$ contains a path $q_1,\dots,q_t$, with $q_1 = c$, $q_t = y$, and $q_2,\dots,q_t \in Y$ (so, $q_2,\dots,q_t \in B)$. Now, suppose that there is a path $p_1,\dots,p_s$ in $G_A \setminus S$, with $p_1 = x$ and $p_s = c$.
Then $p_1,\dots,p_s,q_2,\dots,q_t$ is a path in $G \setminus S$ between $x$ and $y$, contrary to the fact that $S$ is an $(x,y)$-separator of $G$. So, $S$ is an $(x,c)$-separator of $G_A$. Let $S' \subseteq S$ be a minimal $(x,c)$-separator of $G_A$; since $G_A \in \mathcal{G}_{\mathcal{C}}$, we see that $G_A[S'] = G[S']$ belongs to ${\mathcal{C}}$.
If $S'$ is an $(x,y)$-separator of $G$, then the minimality of $S$ guarantees that $S = S'$, and we are done. So, assume that $S'$ is not an $(x,y)$-separator of $G$. Then there is a path $r_1,\dots,r_k$ in $G \setminus S'$, with $r_1 = x$ and $r_k = y$. Since $x \in A$, $y \in B$, and $A$ is anticomplete to $B$, we see that some internal vertex of $r_1,\dots,r_k$ belongs to $C$; let $i$ be the smallest index  in $\{2,\dots,k-1\}$ such that $r_i \in C$. Since $r_i,c \in C$, and $C$ is a clique, we see that $r_i$ and $c$ are either equal or adjacent. In the former case, $r_1,\dots,r_i$ is a path from $x$ to $c$ in $G_A \setminus S'$, and in the latter case, $r_1,\dots,r_i,c$ is a path from $x$ to $c$ in $G_A \setminus S'$. But neither outcome is possible, since $S'$ is an $(x,c)$-separator of $G_A$.
\end{proof}

Theorem~\ref{thm-GC-clique-gluing} implies the following result on the classes of graphs in which every minimal separator is a union of $k$ cliques.

\begin{corollary} \label{cor-G2-clique-gluing}
For every positive integer $k$, the class $\mathcal{G}_{k}$ is closed under gluing along a clique.
\end{corollary}

\begin{proof}
This follows immediately from \cref{thm-GC-clique-gluing}, the fact that $\mathcal{G}_k = \mathcal{G}_{\mathcal{C}_k}$, and the fact that the class ${\mathcal{C}_k}$ is a hereditary graph class containing all complete graphs.
\end{proof}

Note that \cref{cor-G2-clique-gluing} fails for $\mathcal{G}_0$: the two-edge path $P_3$ is an obvious counterexample.

\section{Forbidden induced minors} \label{sec:ForbIndMinor}

For a class of graphs $\mathcal{C}$, let us denote by $\mathcal{M}_{\mathcal{C}}$ the class of all graphs that do not belong to $\mathcal{C}$, but all of whose proper induced minors do belong to $\mathcal{C}$. 
Note that nonisomorphic graphs in $\mathcal{M}_{\mathcal{C}}$ are incomparable under the induced minor relation. 
If the class $\mathcal{C}$ is closed under induced minors, then clearly, $\mathcal{C}$ is precisely the class of all $\mathcal{M}_{\mathcal{C}}$-induced-minor-free graphs.
In this case, we refer to graphs in $\mathcal{M}_{\mathcal{C}}$ as the \emph{minimal forbidden induced minors} for the class $\mathcal{C}$.
More generally, if $\mathcal{M}$ is a class of graphs such that every graph in $\mathcal{C}$ is $\mathcal{M}$-induced-minor-free, we refer to graphs in $\mathcal{M}$ as \emph{forbidden induced minors} for the class $\mathcal{C}$.

We now consider hereditary graph classes that are closed under edge addition. 
We show that for any such class $\mathcal{C}$, both the class $\mathcal{C}$ and the corresponding class $\mathcal{G}_\mathcal{C}$ are closed under induced minors, and we characterize the class $\mathcal{G}_\mathcal{C}$ in terms of forbidden induced minors. 
For graphs $H_1$ and $H_2$ on disjoint vertex sets, the {\em complete join} of $H_1$ and $H_2$, denoted by $H_1 \vee H_2$, is the graph with vertex set $V(H_1) \cup V(H_2)$ and edge set $E(H_1) \cup E(H_2) \cup \{v_1v_2 \mid v_1 \in V(H_1), v_2 \in V(H_2)\}$.

\begin{theorem} \label{thm:GC-induced-minors}
Let $\mathcal{C}$ be a hereditary class of graphs, closed under edge addition. Then both $\mathcal{C}$ and $\mathcal{G}_{\mathcal{C}}$ are closed under induced minors, and $\mathcal{G}_{\mathcal{C}}$ is precisely the class of all $\{2K_1 \vee H \mid H \in \mathcal{M}_{\mathcal{C}}\}$-induced-minor-free graphs.
\end{theorem}

\begin{proof}
Let $\mathcal{C}$ be a hereditary class of graphs, closed under edge addition.
First, we prove that $\mathcal{C}$ is closed under induced minors. 
It suffices to show that $\mathcal{C}$ is closed under vertex deletion and edge contraction. The former follows immediately from the fact that $\mathcal{C}$ is hereditary. Let us prove the latter. Fix a graph $G \in \mathcal{C}$, and fix an edge $xy \in E(G)$. Then $G/xy$ is isomorphic to the graph obtained from $G$ by first deleting $y$, and then adding edges between $x$ and all vertices in $N_G[y] \setminus N_G[x]$; since $\mathcal{C}$ is closed under vertex deletion and edge addition, it follows that $G/xy$ belongs to $\mathcal{C}$. So, $\mathcal{C}$ is closed under edge contraction. This proves that $\mathcal{C}$ is closed under induced minors. 

By Corollary~\ref{prop-GC-hereditary}, $\mathcal{G}_{\mathcal{C}}$ is hereditary. So, in order to prove that $\mathcal{G}_{\mathcal{C}}$ is closed under induced minors, it suffices to show that $\mathcal{G}_{\mathcal{C}}$ is closed under edge contractions. Fix $G \in \mathcal{G}_{\mathcal{C}}$, let $xy$ be an edge of $G$, and set $G' := G/xy$; the vertex of $G'$ to which the edge $xy$ is contracted will be denoted by $v^{xy}$. We must show that $G' \in \mathcal{G}_{\mathcal{C}}$, i.e., that for any minimal separator $S$ of $G'$, we have that $G'[S] \in \mathcal{C}$.

We first deal with minimal separators of $G'$ that contain $v^{xy}$. So, suppose that $S \subseteq V(G')$ is a minimal separator of $G'$ such that $v^{xy} \in S$; we must show that $G'[S] \in \mathcal{C}$. Fix distinct $a,b \in V(G') \setminus S$ such that $S$ is a minimal $(a,b)$-separator of $G'$.
Then $S^* := (S \setminus \{v^{xy}\}) \cup \{x,y\}$ is an $(a,b)$-separator of $G$. Let $S' \subseteq S^*$ be a minimal $(a,b)$-separator of $G$; since $G \in \mathcal{G}_{\mathcal{C}}$, we have that $G[S'] \in \mathcal{C}$.
If $x,y \notin S'$, then $S' \subsetneqq S$ is an $(a,b)$-separator of $G'$, contrary to the minimality of $S$. So, $S'$ contains at least one of $x,y$.
Then $(S' \setminus \{x,y\}) \cup \{v^{xy}\}$ is an $(a,b)$-separator of $G'$; since $(S' \setminus \{x,y\}) \cup \{v^{xy}\} \subseteq S$, the minimality of $S$ implies that $S = (S' \setminus \{x,y\}) \cup \{v^{xy}\}$.
As we show next, the graph $G'[S]$ can be obtained from an induced subgraph of $G[S']$ by possibly adding some edges.
By symmetry, we may assume that $x\in S'$.
Since $S\setminus\{v^{xy}\} = S'\setminus\{x,y\}$, the graph $G'[S]$ is isomorphic to the graph obtained from the subgraph of $G[S']$ induced by $S'\setminus\{y\}$ by adding to it the edges from $x$ to all vertices in $S'\setminus\{x,y\}$ that are adjacent in $G$ to $y$ but not to $x$.
Since $G[S'] \in \mathcal{C}$, and $\mathcal{C}$ is hereditary and closed under edge addition, we deduce that $G'[S] \in \mathcal{C}$, and we are done.

We still have to consider minimal separators of $G'$ that do not contain $v^{xy}$. Here, we first observe that for any pair of nonadjacent vertices $a,b$ of $G'$, and any set $S \subseteq V(G') \setminus \{a,b,v^{xy}\} = V(G) \setminus \{a,b,x,y\}$, both the following hold:
\begin{itemize}
\item[(1)] if $v^{xy} \notin \{a,b\}$, then $S$ is an $(a,b)$-separator of $G$ if and only if $S$ is an $(a,b)$-separator of $G'$;
\item[(2)] if $v^{xy} = a$, then $S$ is an $(x,b)$-separator of $G$ if and only if $S$ is an $(a,b)$-separator of $G'$.
\end{itemize}
Clearly, (1) and (2) imply that any set $S \subseteq V(G') \setminus \{v^{xy}\} = V(G) \setminus \{x,y\}$ is a minimal separator of $G'$ if and only if it is a minimal separator of $G$. But for any $S \subseteq V(G') \setminus \{v^{xy}\} = V(G) \setminus \{x,y\}$, we have that $G'[S] = G[S]$, and moreover, if $S$ is a minimal separator of $G$, then $G[S] \in \mathcal{C}$.
This shows that if a set $S \subseteq V(G') \setminus \{v^{xy}\}$ is a minimal separator of $G'$, then $G'[S] \in \mathcal{C}$.

Finally, it remains to show that $\mathcal{G}_{\mathcal{C}}$ is precisely the class of all $\{2K_1 \vee H \mid H \in \mathcal{M}_{\mathcal{C}}\}$-induced-minor-free graphs. 
Let us first show that all graphs in $\mathcal{G}_{\mathcal{C}}$ are $\{2K_1 \vee H \mid H \in \mathcal{M}_{\mathcal{C}}\}$-induced-minor-free. Since $\mathcal{G}_{\mathcal{C}}$ is closed under induced minors, it suffices to show that, for all $H \in \mathcal{M}_{\mathcal{C}}$, the graph $2K_1 \vee H$ does not belong to $\mathcal{G}_{\mathcal{C}}$. So, fix $H \in \mathcal{M}_{\mathcal{C}}$, and let $x$ and $y$ be the two vertices of the $2K_1$ from $2K_1 \vee H$. 
Then $V(H)$ is a minimal $(x,y)$-separator of $2K_1 \vee H$. 
Since the subgraph of $2K_1 \vee H$ induced by $V(H)$ is $H$, which does not belong to $\mathcal{C}$ (because $H \in \mathcal{M}_{\mathcal{C}}$), it follows that $2K_1 \vee H$ does not belong to $\mathcal{G}_{\mathcal{C}}$. 

For the reverse direction, fix any $\{2K_1 \vee H \mid H \in \mathcal{M}_{\mathcal{C}}\}$-induced-minor-free graph $G$; we must show that $G \in \mathcal{G}_{\mathcal{C}}$. Fix distinct, nonadjacent vertices $a$ and $b$ of $G$, and fix a minimal $(a,b)$-separator $S$ of $G$. We must show that $G[S] \in \mathcal{C}$. Suppose otherwise. Then since $\mathcal{C}$ is closed under induced minors, there exists some $H \in \mathcal{M}_{\mathcal{C}}$ such that $H$ is an induced minor of $G[S]$. We will derive a contradiction by showing that $2K_1 \vee H$ is an induced minor of $G$. Let $\{X_v\}_{v \in V(H)}$ be a family of nonempty, pairwise disjoint subsets of $S$, each inducing a connected subgraph of $G[S]$, and having the property that for all distinct $u,v \in V(H)$, there is an edge between $X_u$ and $X_v$ in $G[S]$ if and only if $uv \in E(H)$. Next, let $X_a$ (resp., $X_b$) be the vertex set of the component of $G \setminus S$ that contains $a$ (resp., $b$). Obviously, $X_a$ and $X_b$ are disjoint and anticomplete to each other in $G$. 
Further, since $S$ is a minimal $(a,b)$-separator of $G$, we see that, in $G$, every vertex of $S$ has a neighbor both in $X_a$ and in $X_b$. In particular, for all $v \in V(H)$, there is an edge between $X_a$ and $X_v$ in $G$, and there is also an edge between $X_b$ and $X_v$ in $G$. But now by considering the family $\{X_v\}_{v \in \{a,b\} \cup V(H)}$, we see that $2K_1 \vee H$ is an induced minor of $G$ (here, $a$ and $b$ are the two vertices of the $2K_1$). 
\end{proof}

We now apply \Cref{thm:GC-induced-minors} to the cases when $\mathcal{C} = \mathcal{C}_k$ for $k\in\{0,1\}$, and thus $\mathcal{G}_{\mathcal{C}} = \mathcal{G}_k$.

For $k = 0$, we have $\mathcal{M}_{\mathcal{C}_0} = \{K_1\}$ and therefore
$\{2K_1 \vee H \mid H \in \mathcal{M}_{\mathcal{C}_0}\} = \{2K_1 \vee K_1\} = \{P_3\}$; we obtain that $\mathcal{G}_0$ is precisely the class of all $P_3$-induced-minor-free graphs.
In particular, we recover the known easy fact that the class $\mathcal{G}_0$ of all disjoint unions of complete graphs is precisely the class of all $P_3$-induced-minor-free graphs (which is also the class of all $P_3$-free graphs).
In this case, since the set of forbidden induced minors is a singleton, it is in fact also the set of \emph{minimal} forbidden induced minors, that is, $\mathcal{M}_{\mathcal{G}_{0}} = \{P_3\}$.

For $k = 1$, we have $\mathcal{M}_{\mathcal{C}_1} = \{2K_1\}$ and we obtain that $\{2K_1 \vee H \mid H \in \mathcal{M}_{\mathcal{C}_1}\} = \{2K_1 \vee 2K_1\} = \{C_4\}$; that is, $\mathcal{G}_1$ is precisely the class of all $C_4$-induced-minor-free graphs.
Note that a graph is $C_4$-induced-minor-free if and only if it contains no induced cycles of length greater than three, that is, it is a chordal graph.
By the definition of ${\cal G}_1$, graphs in ${\cal G}_1$ are precisely the graphs in which all minimal separators are cliques. 
Thus, we have again recovered a known result: a graph is chordal if and only if all its minimal separators are cliques (see~\cite{D61}).
Since the set of forbidden induced minors is a singleton, we again conclude that $\mathcal{M}_{\mathcal{G}_{1}} = \{C_4\}$.

The reader may wonder whether, in \cref{thm:GC-induced-minors}, it is necessary to assume that $\mathcal{C}$ is closed under edge addition. Here, we note that if $\mathcal{C}$ is the class of all edgeless graphs, then $K_{2,3} \in \mathcal{G}_{\mathcal{C}}$, but contracting one edge of $K_{2,3}$ produces the diamond, which does not belong to $\mathcal{G}_{\mathcal{C}}$. Thus, the assumption about edge additions cannot be simply removed from \cref{thm:GC-induced-minors}, although it is possible that some other (weaker) assumption would suffice instead. 

\begin{corollary}\label{cor-K2k1} For all integers $k\ge 0$, classes $\mathcal{C}_k$ and $\mathcal{G}_k$ are closed under induced minors, and all graphs in $\mathcal{G}_k$ are $K_{2,k+1}$-induced-minor-free. 
\end{corollary} 
\begin{proof} 
Fix an integer $k\ge 0$. 
Obviously, $\mathcal{C}_k$ is hereditary and closed under edge addition. So, by Theorem~\ref{thm:GC-induced-minors}, $\mathcal{C}_k$ is closed under induced minors, and $\mathcal{G}_k = \mathcal{G}_{\mathcal{C}_k}$ is closed under induced minors. It remains to show that $K_{2,k+1} \notin \mathcal{G}_k$. But this is obvious: one minimal separator of $K_{2,k+1}$ induces an edgeless subgraph on $k+1$ vertices in $K_{2,k+1}$, and $\mathcal{C}_k$ does not contain edgeless graphs on more than $k$ vertices. 
\end{proof} 

We now proceed to showing that when $\mathcal{C}$ is is a hereditary class of graphs closed not only under edge addition, but also under the addition of universal vertices, the conclusion of \cref{thm:GC-induced-minors} can be strengthened to a characterization of the class of minimal forbidden induced minors for the class $\mathcal{G}_\mathcal{C}$ (see \Cref{thm-min-forb-ind-minor}).
We start with some preparatory statements.
A nonnull graph is {\em anticonnected} if its complement is connected. 
An {\em anticomponent} of a nonnull graph $G$ is a maximal anticonnected induced subgraph of $G$. Clearly, every nonnull graph is the complete join of its anticomponents.

\begin{proposition} \label{prop-H-anticonn-ind-minor} Let $H$ be an anticonnected graph on at least two vertices, and assume that $H$ is an induced minor of a graph $G$. Let $\{X_v\}_{v \in V(H)}$ be an induced minor model of $H$ in $G$.
Then there exists an anticomponent $C$ of $G$ such that $\bigcup_{v \in V(H)} X_v \subseteq V(C)$. 
\end{proposition} 
\begin{proof} 
$\phantom{i}$ 

\begin{quote} 
\textbf{Claim.} For all $v \in V(H)$, there exists an anticomponent $C$ of $G$ such that $X_v \subseteq V(C)$. 
\end{quote} 
{\em Proof of the Claim.} Suppose otherwise. Then there exists some vertex $u \in V(H)$ and distinct anticomponents $C_1$ and $C_2$ of $G$ such that $X_u$ intersects both $V(C_1)$ and $V(C_2)$. Let us show that $u$ is a universal vertex of $H$. 
Fix any $w \in V(H) \setminus \{u\}$. Let $C$ be any anticomponent of $G$ such that $X_w \cap V(C) \neq \emptyset$. By symmetry, we may assume that $C \neq C_1$.\footnote{Indeed, either $C \neq C_1$ or $C \neq C_2$, and by symmetry, we may assume that $C \neq C_1$.} Then $X_u \cap V(C_1)$ and $X_w \cap V(C)$ are both nonempty and complete to each other in $G$, and in particular, there is at least one edge between $X_u$ and $X_w$ in $G$. So, $uw \in E(H)$. This proves that $u$ is a universal vertex of $H$. But this is impossible, since $H$ is an anticonnected graph on at least two vertices, and consequently, $H$ has no universal vertices.~$\blacklozenge$ 

\medskip 

Fix any anticomponent $C$ of $G$ such that $\Big(\bigcup_{v \in V(H)} X_v\Big) \cap V(C) \neq \emptyset$, and set $U := \{v \in V(H) \mid X_v \cap V(C) \neq \emptyset\}$. By construction, we have that $U \neq \emptyset$, and by the claim, we have that $U = \{v \in V(H) \mid X_v \subseteq V(C)\}$. It now suffices to show that $U = V(H)$, for it will then follow that $\bigcup_{v \in V(H)} X_v \subseteq V(C)$, which is what we need. Since $U \neq \emptyset$ and $H$ is anticonnected, it is in fact enough to show that $U$ is complete to $V(H) \setminus U$ in $H$. So, fix some $u \in U$ and $w \in V(H) \setminus U$. Since $u \in U$, we have that $X_u \subseteq V(C)$. On the other hand, by the claim, there exists an anticomponent $D$ of $G$ such that $X_w \subseteq V(D)$; since $w \notin U$, we have that $D \neq C$. Since $C$ and $D$ are distinct anticomponents of $G$, we know that $V(C)$ and $V(D)$ are complete to each other in $G$; consequently, $X_u$ is complete to $X_w$ in $G$, and it follows that $uw \in E(H)$. This proves that $U$ is complete to $V(H) \setminus U$, and we are done. 
\end{proof} 

\begin{proposition} \label{prop-remove-2K1} Let $H_1$ and $H_2$ be graphs. Assume that $H_1$ contains no universal vertices and that $2K_1 \vee H_1$ is an induced minor of $2K_1 \vee H_2$. Then $H_1$ is an induced minor of $H_2$. 
\end{proposition} 

\begin{proof} 
If $H_1$ is the null graph, then it is obviously an induced minor of $H_2$. So, we may assume that $H_1$ is nonnull. Since $2K_1 \vee H_1$ is an induced minor of $2K_1 \vee H_2$, it follows that the graph $H_2$ is also nonnull. 

Now, using the fact that $2K_1 \vee H_1$ is an induced minor of $2K_1 \vee H_2$, we fix an induced minor model $\{X_v\}_{v \in V(2K_1 \vee H_1)}$ of $2K_1 \vee H_1$ in $2K_1 \vee H_2$.
If $\bigcup_{v \in V(H_1)} X_v \subseteq V(H_2)$, then $H_1$ is an induced minor of $H_2$, and we are done. So, we may assume that $\bigcup_{v \in V(H_1)} X_v \not\subseteq V(H_2)$. Fix an anticomponent $H$ of $H_1$ such that $\bigcup_{v \in V(H)} X_v \not\subseteq V(H_2)$. Since $H_1$ has no universal vertices, we have that $|V(H)| \geq 2$. In view of Proposition~\ref{prop-H-anticonn-ind-minor}, it follows that $\bigcup_{v \in V(H)} X_v \subseteq V(2K_1)$. Since $H$ has at least two vertices, it follows that $H \cong 2K_1$ and $\bigcup_{v \in V(H)} X_v = V(2K_1)$. 
Consequently, $\bigcup_{v \in V(2K_1 \vee H_1) \setminus V(H)} X_v\subseteq V(H_2)$, and see that $(2K_1 \vee H_1) \setminus V(H)$ is an induced minor of $H_2$. 
But note that $(2K_1 \vee H_1) \setminus V(H) \cong H_1$. So, $H_1$ is an induced minor of $H_2$. 
\end{proof} 

We note that the assumption that $H_1$ contains no universal vertices cannot be removed from Proposition~\ref{prop-remove-2K1}. To see this, note that $2K_1 \vee K_2$ is an induced minor of $2K_1 \vee 3K_1$ (indeed, we obtain $2K_1 \vee K_2$ by contracting any one edge of $2K_1 \vee 3K_1$), but $K_2$ is not an induced minor of $3K_1$. 

\begin{lemma} \label{lemma-GC-incomparable} Let $\mathcal{C}$ be a class of graphs, closed under the addition of universal vertices. Then no graph in $\mathcal{M}_{\mathcal{C}}$ contains a universal vertex. Moreover, for any two nonisomorphic graphs $H_1,H_2 \in \mathcal{M}_{\mathcal{C}}$, the graphs $2K_1 \vee H_1$ and $2K_1 \vee H_2$ are incomparable with respect to the induced minor relation. 
\end{lemma} 
\begin{proof} 
Let us first show that no graph in $\mathcal{M}_{\mathcal{C}}$ contains a universal vertex. Suppose otherwise, and fix a graph $H \in \mathcal{M}_{\mathcal{C}}$ that contains a universal vertex $u$. By the definition of $\mathcal{M}_{\mathcal{C}}$, we have that $H \setminus u$ belongs to $\mathcal{C}$. But then since $\mathcal{C}$ is closed under the addition of universal vertices, we have that $H \in \mathcal{C}$, contrary to the fact that $H \in \mathcal{M}_{\mathcal{C}}$. 

Now, fix any two nonisomorphic graphs $H_1,H_2 \in \mathcal{M}_{\mathcal{C}}$. 
By the definition of $\mathcal{M}_{\mathcal{C}}$, the graphs $H_1$ and $H_2$ are incomparable with respect to the induced minor relation. Moreover, by what we just proved, $H_1$ and $H_2$ have no universal vertices. So, by Proposition~\ref{prop-remove-2K1}, $2K_1 \vee H_1$ and $2K_1 \vee H_2$ are incomparable with respect to the induced minor relation. 
\end{proof}

\begin{theorem} \label{thm-min-forb-ind-minor} 
Let $\mathcal{C}$ be a hereditary class of graphs, closed under edge addition, and closed under the addition of universal vertices. Then $\mathcal{M}_{\mathcal{G}_{\mathcal{C}}} = \{2K_1 \vee H \mid H \in \mathcal{M}_{\mathcal{C}}\}$. 
\end{theorem} 
\begin{proof} 
By Theorem~\ref{thm:GC-induced-minors}, $\mathcal{G}_{\mathcal{C}}$ is precisely the class of all $\{2K_1 \vee H \mid H \in \mathcal{M}_{\mathcal{C}}\}$-induced-minor-free graphs. On the other hand, Lemma~\ref{lemma-GC-incomparable} guarantees that nonisomorphic graphs in $\{2K_1 \vee H \mid H \in \mathcal{M}_{\mathcal{C}}\}$ are incomparable with respect to the induced minor relation. So, $\mathcal{M}_{\mathcal{G}_{\mathcal{C}}} = \{2K_1 \vee H \mid H \in \mathcal{M}_{\mathcal{C}}\}$. 
\end{proof}

We now apply \Cref{thm-min-forb-ind-minor} to the cases when $\mathcal{C} = \mathcal{C}_2$, and thus $\mathcal{G}_{\mathcal{C}} = \mathcal{G}_2$.
It can be shown that $\mathcal{M}_{\mathcal{G}_2} = \{\overline{K_2 \cup C_{2k+1}} \mid k \in \mathbb{N}\}$ (see \cref{fig:MM2}).
To this end, the following auxiliary proposition will be useful.

\begin{figure}
\begin{center}
\includegraphics[scale=0.65]{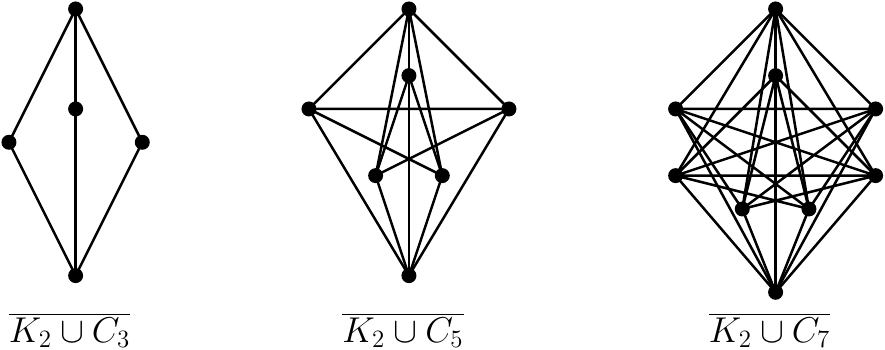}
\end{center}
\caption{Some small graphs in $\mathcal{M}_{\mathcal{G}_2}$.} \label{fig:MM2}
\end{figure}

\begin{proposition} \label{prop-ind-minor-complement-sg} If a graph $H_1$ is an induced minor of a graph $H_2$, then $\overline{H_1}$ is isomorphic to a (not necessarily induced) subgraph of $\overline{H_2}$. 
\end{proposition} 
\begin{proof} 
The result follows from the following technical claim via simple induction. 

\begin{quote} 
{\em Claim.} If a graph $H$ is obtained from a graph $G$ by deleting one vertex or contracting one edge, then $\overline{H}$ is isomorphic to a subgraph of $\overline{G}$. 
\end{quote} 
{\em Proof of the Claim.} Fix graphs $H$ and $G$, and assume that $H$ is obtained from $G$ by deleting one vertex or contracting one edge. We must show that $\overline{H}$ is isomorphic to a (not necessarily induced) subgraph of $\overline{G}$. If $H$ is obtained from $G$ by deleting one vertex, then this is obvious. So, assume that $H$ is obtained from $G$ by contracting an edge $xy$ of $G$. Then $\overline{H}$ is isomorphic to the graph obtained from $\overline{G}$ by first deleting $y$ and then deleting all edges between $x$ and $N_G[y] \setminus N_G[x] = N_{\overline{G}}(x) \setminus N_{\overline{G}}(y)$. So, $\overline{H}$ is isomorphic to a subgraph of $\overline{G}$.~$\blacklozenge$  
\end{proof} 

\begin{corollary}\label{cor:GG2}
$\mathcal{M}_{\mathcal{G}_2} = \{\overline{K_2 \cup C_{2k+1}} \mid k \in \mathbb{N}\}$. 
\end{corollary}

\begin{proof} 
Note that for all positive integers $k$, we have that $\overline{K_2 \cup C_{2k+1}} \cong 2K_1 \vee \overline{C_{2k+1}}$. So, in view of Theorem~\ref{thm-min-forb-ind-minor}, it is enough to show that $\mathcal{M}_{\mathcal{C}_2} = \{\overline{C_{2k+1}} \mid k \in \mathbb{N}\}$. Note that graphs in $\mathcal{C}_2$ are precisely the complements of bipartite graphs, and it is well known that a graph is bipartite if and only if it contains no odd cycle as an induced subgraph. So, $\mathcal{C}_2$ is precisely the class of all $\{\overline{C_{2k+1}} \mid k \in \mathbb{N}\}$-free graphs. Since $\mathcal{C}_2$ is closed under induced minors (by Corollary~\ref{cor-K2k1}), it follows that $\mathcal{C}_2$ is in fact the class of all $\{\overline{C_{2k+1}} \mid k \in \mathbb{N}\}$-induced-minor-free graphs. It remains to show that (nonisomorphic) graphs in $\{\overline{C_{2k+1}} \mid k \in \mathbb{N}\}$ are incomparable with respect to the induced minor relation. But this follows from Proposition~\ref{prop-ind-minor-complement-sg}, and from the fact that no cycle is a subgraph of a cycle of different length. 
\end{proof}

A graph is \emph{$1$-perfectly-orientable} if it admits an orientation in which the out-neighborhood of each vertex is a clique of the underlying graph.
It was shown by Hartinger and Milani{\v{c}} in~\cite{MR3647815} that all $1$-perfectly-orientable graphs are $\{\overline{K_2 \cup C_{2k+1}} \mid k \in \mathbb{N}\}$-induced-minor-free.
Thus, \cref{cor:GG2} implies that $1$-perfectly-orientable graphs form a subclass of $\mathcal{G}_2$.
Let us also remark that the proof of the mentioned result in~\cite{MR3647815} also gives a proof of the fact that nonisomorphic graphs in  $\{\overline{K_2 \cup C_{2k+1}} \mid k \in \mathbb{N}\}$ are incomparable with respect to the induced minor relation, which, when combined with \Cref{thm-min-forb-ind-minor}, gives an alternative proof of \Cref{cor:GG2}.

\subsection{Algorithmic considerations}

It was shown by Dallard et al.~in~\cite{TreeDecompBoundAlpha} that, for each positive integer $k$, the \textsc{Maximum Weight Stable Set} problem can be solved in $\mathcal{O}(n^{2k})$ time for $n$-vertex $K_{2,k}$-induced-minor-free graphs.
To connect this result with the classes $\mathcal{G}_k$, by \cref{cor-K2k1} every graph in $\mathcal{G}_k$ is $K_{2,k+1}$-induced-minor-free.
Therefore, the result by Dallard et al.~implies the following.

\begin{corollary}\label{MWIS-GG_k}
For each integer $k \geq 0$, the \textsc{Maximum Weight Stable Set} problem can be solved in $\mathcal{O}(n^{2k+2})$ time for $n$-vertex graphs in $\mathcal{G}_k$.
\end{corollary}

Similar results hold for a number of other related problems, including the \textsc{Maximum Induced Matching} problem, the \textsc{Dissociation Set} problem, etc. We refer to~\cite{dallard2021treewidth} for the details.

Furthermore, Dallard et al.~showed in~\cite{MR4334541} that in any class of $K_{2,k}$-induced-minor-free graphs, the treewidth of the graphs in the class is bounded from above by some polynomial function of the clique number (see also~\cite{TreeDecompBoundAlpha}).
Combining this with a result of Chaplick et al.~\cite[Theorem 12]{MR4332111}, it follows that for any two positive integers $k$ and $\ell$, the \textsc{$\ell$-Coloring} problem is solvable in time $\mathcal{O}(n)$ in the class of $n$-vertex $K_{2,k}$-induced-minor-free graphs, and thus in the class $\mathcal{G}_k$ as well.
(The $\mathcal{O}$-notation hides a constant depending on $k$ and $\ell$.)
The same result holds in fact for the more general \textsc{List $\ell$-Coloring} problem, in which every vertex is equipped with a list of available colors from the set $\{1,\ldots, \ell\}$.

\section{Vertex neighborhoods} \label{sec:k-simplicial}

Recall that a vertex $v$ in a graph $G$ is \emph{$k$-simplicial} it its neighborhood is a union of $k$ cliques.
Given a graph $G$ and an integer $k\ge 0$, a \emph{$k$-simplicial elimination ordering} of $G$ is an ordering $v_1,\dots,v_n$ of the vertices of $G$ such that for all $i \in \{1,\dots,n\}$, $v_i$ is $k$-simplicial in the graph $G[v_1,\dots,v_i]$
For $k = 1$, resp.~$k = 2$, a $k$-simplicial elimination ordering is also called a \emph{perfect elimination ordering}, resp.~a \emph{bisimplicial elimination ordering}.

A classical result due to Dirac~\cite{D61} states that every nonnull chordal graph has a simplicial vertex.
An alternative proof of this result was given by Rose, Tarjan, and Lueker~\cite{rose1976algorithmic}, who showed that every LexBFS ordering of a chordal graph is a perfect elimination ordering.
In this section we generalize these results by showing that for every positive integer $k$, every LexBFS ordering of a graph in $\mathcal{G}_k$ is a $k$-simplicial elimination ordering.
In particular, this shows that every nonnull graph in $\mathcal{G}_k$ has a $k$-simplicial vertex.
We also examine some algorithmic consequences of these results for the case $k = 2$.

\begin{sloppypar}
For a family $\mathcal{F}$ of graphs, an ordering $v_1,\ldots,v_n$ of the vertices of a graph $G$ is an \emph{\hbox{${\cal F}$-elimination} ordering} if for every index $i \in \{1,\ldots,n\}$, the graph $G[N_{G[v_1,\ldots,v_i]}(v_i)]$ is ${\cal F}$-free.
Note that a graph $G$ admits an ${\cal F}$-elimination ordering if and only if every nonnull induced subgraph of $G$ contains a vertex whose neighborhood induces an ${\cal F}$-free subgraph in $G$.
\end{sloppypar}

LexBFS is a linear-time algorithm of Rose, Tarjan, and Lueker~\cite{rose1976algorithmic} whose input is any nonnull graph $G$ together with a vertex $s \in V(G)$, and whose output is an ordering of the vertices of $G$ starting at $s$.
It is a restricted version of Breadth First Search, where the usual queue of vertices is replaced by a queue of unordered subsets of the vertices, which is sometimes refined, but never reordered (for details, see~\cite{rose1976algorithmic}).
An ordering of the vertices of a graph $G$ is a \emph{LexBFS ordering} if there exists a vertex $s$ of $G$ such that the ordering can be produced by LexBFS when the input is $G,s$.

In certain cases, ${\cal F}$-elimination orderings can be found using LexBFS.
This relies on the concept of locally ${\cal F}$-decomposable graphs and graph classes, introduced by Aboulker et al.~in~\cite{aboulker2015vertex}.
Let ${\cal F}$ be a family of graphs.
A graph $G$ is \emph{locally ${\cal F}$-decomposable} if for every vertex $v$ of $G$, every $F\in {\cal F}$ contained, as an induced subgraph, in $G[N_G(v)]$ and every component $C$ of $G\setminus N_G[v]$, there exists $y\in V(F)$ such that $y$ has a nonneighbor in $F$ and has no neighbors in $C$.
A class of graphs ${\cal G}$ is \emph{locally ${\cal F}$-decomposable} if every graph $G\in {\cal G}$ is a locally ${\cal F}$-decomposable graph.

\begin{theorem}[Aboulker et al.~\cite{aboulker2015vertex}] \label{t1}
If ${\cal F}$ is a family of noncomplete graphs and $G$ is a locally ${\cal F}$-decomposable graph, then every LexBFS ordering of $G$ is an ${\cal F}$-elimination ordering.
\end{theorem}

For a hereditary graph class $\mathcal{C}$, we denote by $\mathcal{F}_{\mathcal{C}}$ the class of all graphs $G$ such that $G$ does not belong to $\mathcal{C}$, but all proper induced subgraphs of $G$ do belong to $\mathcal{C}$.

\begin{theorem} \label{thm-GC-vertex-neighborhood-in-C}
Let $\mathcal{C}$ be a hereditary class of graphs that is closed under the addition of universal vertices.
Then, for every graph $G \in \mathcal{G}_{\mathcal{C}}$, every LexBFS ordering of $G$ is an $\mathcal{F}_{\mathcal{C}}$-elimination ordering.
\end{theorem}

\begin{proof}
First, we note that no graph in $\mathcal{F}_{\mathcal{C}}$ contains a universal vertex (and thus, $\mathcal{F}_{\mathcal{C}}$ is a family of noncomplete graphs). Indeed, if some $F \in \mathcal{F}_{\mathcal{C}}$ contained a universal vertex $u$, then the definition of $\mathcal{F}_{\mathcal{C}}$ would imply that $F \setminus u$ belongs to $\mathcal{C}$, and since $\mathcal{C}$ is closed under the addition of universal vertices, it would follow that $F \in {\mathcal{C}}$, a contradiction. 

Now, fix a graph $G \in \mathcal{G}_{\mathcal{C}}$.
We claim that $G$ is locally $\mathcal{F}_{\mathcal{C}}$-decomposable.
Consider a vertex $x \in V(G)$. 
Suppose that $F$ is an induced subgraph of $G[N_G(x)]$ such that $F \in \mathcal{F}_{\mathcal{C}}$. 
By the above, $F$ does not contain a universal vertex, and consequently, every vertex of $F$ has a nonneighbor in $F$. 
Let $C$ be a component of $G \setminus N_G[x]$; we must show that some vertex in $F$ is anticomplete to $V(C)$. 
Suppose otherwise, that is, suppose that every vertex in $V(F)$ has a neighbor in $V(C)$. Let $z \in V(C)$. Clearly, $N_G(x)$ is an $(x,z)$-separator of $G$, and moreover, any minimal $(x,z)$-separator of $G$ included in $N_G(x)$ includes $V(F)$; since $G \in \mathcal{G}_{\mathcal{C}}$ and $\mathcal{C}$ is hereditary, it follows that $F \in \mathcal{C}$, contrary to the fact that $F \in \mathcal{F}_{\mathcal{C}}$. Thus, some vertex in $V(F)$ is indeed anticomplete to $V(C)$. It follows that $G$ is locally $\mathcal{F}_{\mathcal{C}}$-decomposable, and so by \cref{t1}, every LexBFS ordering of $G$ is an $\mathcal{F}_{\mathcal{C}}$-elimination ordering.
This completes the proof.
\end{proof}

The reader may wonder whether, in \cref{thm-GC-vertex-neighborhood-in-C}, it might be possible to eliminate the hypothesis that $\mathcal{C}$ is closed under the addition of universal vertices.
This would in fact not be possible (at least not without adding some other, perhaps weaker, hypothesis).
To see this, fix any positive integer $\ell$, and any hereditary class $\mathcal{C}$ that does not contain $K_{\ell}$. The class $\mathcal{G}_{\mathcal{C}}$ contains all complete graphs, and in particular, $K_{\ell+1} \in \mathcal{G}_{\mathcal{C}}$. However, the neighborhood of any vertex of $K_{\ell+1}$ induces a $K_{\ell}$, and $K_{\ell} \notin \mathcal{C}$.

\begin{corollary} \label{cor-GGk-simplicial}
Let $k$ be a positive integer.
Then, for every graph $G$ in $\mathcal{G}_k$,
every LexBFS ordering of $G$ is a $k$-simplicial elimination ordering.
\end{corollary}

\begin{proof}
Fix an integer $k\ge 1$.
Recall that $\mathcal{C}_k$ is the class of graphs whose vertex set can be partitioned into $k$ cliques.
Let $G\in \mathcal{C}_k$, and let $C_1,\ldots, C_k$ be $k$ cliques in $G$ partitioning the vertex set of $G$.
If $G'$ is the graph obtained by adding a universal vertex $u$ to $G$, then $C_1\cup\{u\}, C_2, \ldots, C_k$ are $k$ cliques in $G'$ forming a partition of the vertex set of $G'$.
It follows that $\mathcal{C}_k$ is closed under the addition of universal vertices. 
By \cref{thm-GC-vertex-neighborhood-in-C}, for every graph $G$ in $\mathcal{G}_k$, every LexBFS ordering of $G$ is a $k$-simplicial elimination ordering.
\end{proof}

\begin{corollary} \label{cor-GGk-simplicial-2}
For every positive integer $k$, every nonnull graph in $\mathcal{G}_k$ has a $k$-simplicial vertex.
\end{corollary}

\begin{remark}
\Cref{cor-GGk-simplicial} can also be obtained from the fact that every nonnull graph $G$ contains a \emph{moplex}~\cite{moplex}, that is, a clique $C$ such that every two vertices in $C$ have the same closed neighborhood and the neighborhood of $C$ is either empty or a minimal separator of $G$.
(In fact, the last vertex visited by any execution of LexBFS on $G$ necessarily belongs to a moplex.)
Given a graph $G\in \mathcal{G}_k$ and a vertex $v$ that belongs to a moplex $C$ of $G$, there exist $k$ cliques $C_1,\ldots, C_k$ covering the neighborhood of $C$.
But then $C_1,\ldots, C_k\cup(C\setminus\{v\})$ are $k$ cliques covering $N_G(v)$, showing that $v$ is a $k$-simplicial vertex.
\end{remark}

\begin{remark}
For $k>1$, the statement of \Cref{cor-GGk-simplicial} does not generalize to the class of graphs that admit a $k$-simplicial elimination ordering.
To see this, let $G$ be the complete bipartite graph $K_{2,k+1}$.
Then, $G$ admits a bisimplicial elimination ordering obtained by placing the two vertices of degree $k+1$ before all the $k+1$ the vertices of degree $2$ in the ordering.
On the other hand, any LexBFS ordering of $G$ starting at a vertex with degree $k+1$ will end in the other vertex of degree $k+1$, which is not a $k$-simplicial vertex in $G$.
\end{remark}

\subsection{Algorithmic implications}

The \textsc{Maximum Weight Clique} problem can be solved in polynomial time for $n$-vertex graphs that admit a bisimplicial elimination ordering, see~\cite{ElimGraphs}.
The algorithm iteratively removes bisimplicial vertices and reduces the problem to solving $n$ instances of the \textsc{Maximum Weight Stable Set} problem in bipartite graphs.
The polynomial running time of the algorithm given in~\cite{ElimGraphs} was based on polynomial-time solvability of the \textsc{Maximum Weight Stable Set} problem in the class of perfect graphs.
Using maximum flow techniques, an improved running time of $\mathcal{O}(n^4)$ can be achieved, see~\cite{MR4357319}.
For graphs in $\mathcal{G}_2$, a further improvement can be obtained using LexBFS and recent developments on maximum flow algorithms.

\begin{theorem}\label{MWC-G22}
For every $\epsilon>0$, the \textsc{Maximum Weight Clique} problem can be solved in $\mathcal{O}(n^{3+\epsilon})$ time for $n$-vertex graphs in $\mathcal{G}_2$.
\end{theorem}
\begin{proof}
Let $G$ be an $n$-vertex graph in $\mathcal{G}_2$.
In time  $\mathcal{O}(n^{2})$, we compute a LexBFS ordering $v_1,\ldots, v_n$ of $G$.
By \cref{cor-GGk-simplicial}, $v_1,\ldots, v_n$ is a bisimplicial elimination ordering of $G$.
For each $i\in \{1,\ldots, n\}$, let $G_i$ be the graph induced by the closed neighborhood of $v_i$ in the graph $G[v_1,\dots,v_i]$. 
We will show that for each $i\in \{1,\ldots, n\}$, we can compute a maximum-weight clique $C_i$ in $G_i$ in time $\mathcal{O}(n^{2+\epsilon})$.
This will suffice, since a clique $C_i$ with maximum total weight is also a clique in $G$ with maximum total weight.

Fix an $i\in \{1,\ldots, n\}$.
A maximum-weight clique in the graph $G_i$ consists of $v_i$ and a maximum-weight stable set in the complement of the graph $G_i-v_i$.
Since $v_i$ is a bisimplicial vertex in the graph $G[v_1,\dots,v_i]$, the complement of the graph $G_i-v_i$ is bipartite.
The problem of finding a maximum-weight stable set in a vertex-weighted bipartite graph can be reduced in linear time to a maximum flow problem in a derived network (see, e.g.,~\cite{MR2221557}).
Using a recent result due to Chen et al.~\cite{MR4537240} showing that the maximum flow problem can be solved in almost linear time, we conclude that a maximum-weight independent set in the complement of the graph $G_i-v_i$ can be computed in time $\mathcal{O}(n^{2+\epsilon}).$
The claimed $\mathcal{O}(n^{3+\epsilon})$ overall time complexity follows.
\end{proof}

As we will show in the next section (more specifically in \cref{thm-max-clique-\NP-hard}), this result cannot be generalized to graphs in $\mathcal{G}_k$ for $k\ge 3$,  unless P = NP.

\section{\NP-hardness results for $\mathcal{G}_k$, $k \geq 3$} \label{sec:recGGK}

In this section we prove that for all $k \geq 3$, it is \NP-hard to recognize graphs in $\mathcal{G}_k$.
The time complexity of recognizing graphs in $\mathcal{G}_2$ is still unknown.
We also show that the \textsc{Maximum Clique} problem is \NP-hard for $\mathcal{G}_3$ (and consequently for $\mathcal{G}_k$ whenever $k \geq 3$).

Recall that for an integer $k$, the \textsc{$k$-Coloring} problem is the problem of determining whether the input graph is $k$-colorable.
It is well known that for any integer $k \geq 3$, the \textsc{$k$-Coloring} problem is \NP-complete; we prove the \NP-hardness of recognizing graphs in $\mathcal{G}_k$ ($k \geq 3$) by a reduction from this problem.
We begin with a technical proposition.

\begin{proposition} \label{prop-kCol-red}
Let $k \geq 0$ be an integer, and let $G$ be a graph.
Let $G'$ be the graph obtained from $G$ by adding two new, nonadjacent vertices, and making them adjacent to all vertices of $G$.
Then $\chi(\overline{G}) \leq k$ if and only if $G' \in \mathcal{G}_k$.
\end{proposition}

\begin{proof}
Let $a$ and $b$ be the two vertices added to $G$ to form $G'$.

Suppose first that $G' \in \mathcal{G}_k$. Clearly, $V(G)$ is a minimal $(a,b)$-separator of $G'$, and so $V(G)$ is a union of $k$ cliques of $G'$. But then $\chi(\overline{G}) \leq k$.

Suppose now that $\chi(\overline{G}) \leq k$; we must show that $G' \in \mathcal{G}_k$.
Fix two distinct, nonadjacent vertices $x,y \in V(G')$ and let $S$ be a minimal $(x,y)$-separator of $G'$.
We must show that $S$ is a union of $k$ cliques of $G'$.

Suppose first that $\{x,y\} \cap \{a,b\} \neq \emptyset$. By symmetry, we may assume that $x = a$. Since $b$ is the only nonneighbor of $a$ in $G'$, it follows that $y = b$. Since $\{a,b\} = \{x,y\}$ is complete to $V(G) = V(G') \setminus \{a,b\}$, it follows that $V(G)$ is the only $(x,y)$-separator of $G'$. So, $S = V(G)$. Since $\chi(\overline{G}) \leq k$, it follows that $S$ is a union of $k$ cliques of $G'$.

From now on, we assume that $\{x,y\} \cap \{a,b\} = \emptyset$, so that $x,y \in V(G)$. 
Since $x$ and $y$ are nonadjacent, we see that $V(G)$ is not a clique, and consequently $k\ge \chi(\overline{G})\ge 2$. 
Note that $\{a,b\}$ is complete to $\{x,y\}$, and so $a,b \in S$. Now, $\chi(\overline{G}) \leq k$, and so $S \setminus \{a,b\}$ is a union of $k$ cliques of $G$, say $C_1,\dots,C_k$. 
Using the fact that $\{a,b\}$ is complete to $V(G)$ in $G'$, and the fact that $k \geq 2$, we see that $S$ is a union of $k$ cliques of $G'$, namely $C_1 \cup \{a\},C_2 \cup \{b\},C_3,\dots,C_k$.

We have now shown that $G' \in \mathcal{G}_k$, and we are done.
\end{proof}

\begin{theorem} \label{thm-GGk-rec-\NP-hard} For every integer $k \geq 3$, it is \NP-hard to recognize graphs in $\mathcal{G}_k$.
\end{theorem}
\begin{proof}
Fix an integer $k \geq 3$, and let $G$ be any graph. We form a graph $G'$ by adding two new, nonadjacent vertices to $\overline{G}$, and making them adjacent to all vertices of $\overline{G}$. Since $\overline{\overline{G}} = G$, \cref{prop-kCol-red} guarantees that $\chi(G) \leq k$ if and only if $G' \in \mathcal{G}_k$. Since \textsc{$k$-Coloring} is \NP-complete, it follows that recognizing graphs in $\mathcal{G}_k$ is \NP-hard.
\end{proof}

\begin{theorem} \label{thm-max-clique-\NP-hard} The \textsc{Maximum Clique} problem is \NP-hard for graphs in $\mathcal{G}_3$.
\end{theorem}

\begin{proof}
Note that $\mathcal{G}_3$ contains all graphs whose vertex set can be partitioned into three cliques; moreover, note that the vertex set of a graph can be partitioned into three cliques if and only if the complement of the graph is $3$-colorable. 
Thus, it suffices to show that the \textsc{Maximum Stable Set} problem is \NP-hard for 3-colorable graphs. But this readily follows from~\cite{poljak}. Indeed, as observed by Poljak~\cite{poljak}, for any graph $G$, the graph $G^*$ obtained from $G$ by subdividing each edge twice has the property that $\alpha(G^*) = \alpha(G)+|E(G)|$. But notice that for any graph $G$, the graph $G^*$ is 3-colorable. Thus, since the \textsc{Maximum Stable Set} problem is \NP-hard for general graphs, it is \NP-hard for 3-colorable graphs.
\end{proof}

\section{Diamond-free graphs in $\mathcal{G}_2$} \label{sec:diamond}

We remind the reader that the \emph{diamond} is the four-vertex graph obtained from the complete graph $K_4$ by deleting one edge (see \cref{fig:diamond}).
In this section, we prove a decomposition theorem for diamond-free graphs in $\mathcal{G}_2$, which implies a polynomial-time recognition algorithm for this class of graphs.
We begin with some definitions.

\begin{figure}
\begin{center}
\includegraphics[scale=0.6]{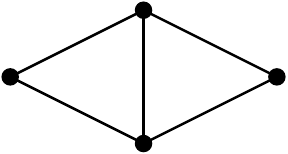}
\end{center}
\caption{The diamond.} \label{fig:diamond}
\end{figure}

\begin{figure}
\begin{center}
\includegraphics[scale=0.65]{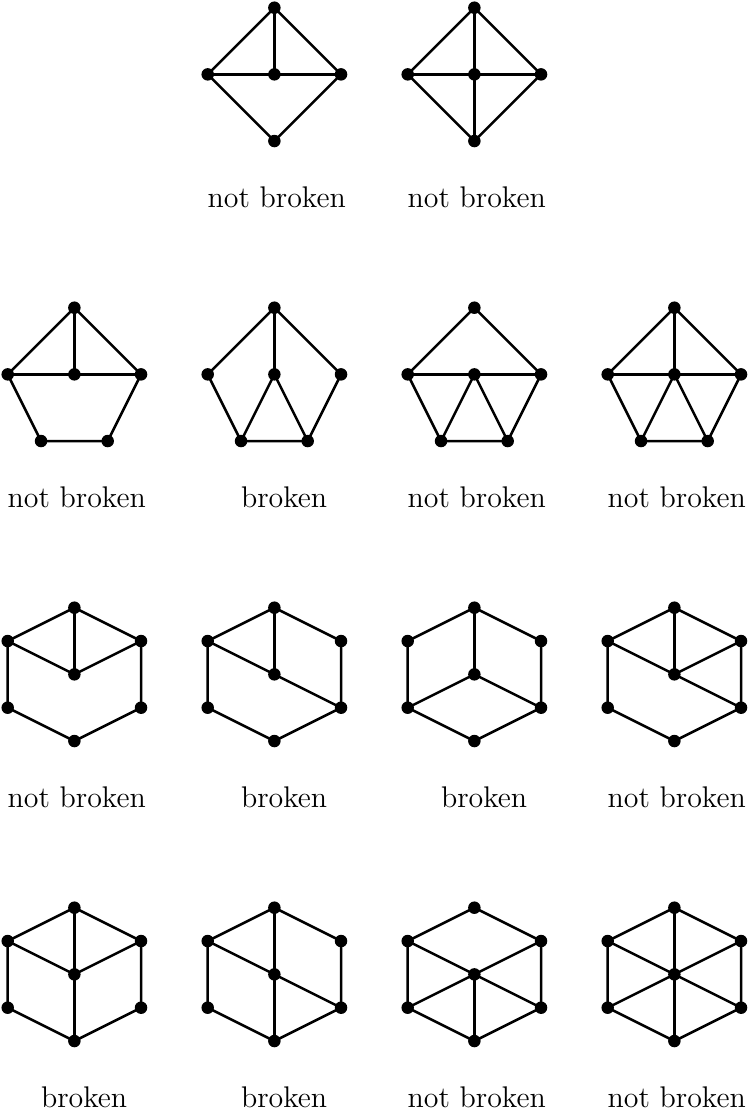}
\end{center}
\caption{Some small wheels, classified as broken or not broken.} \label{fig:BrokenWheel}
\end{figure}

\begin{figure}
\begin{center}
\includegraphics[scale=0.6]{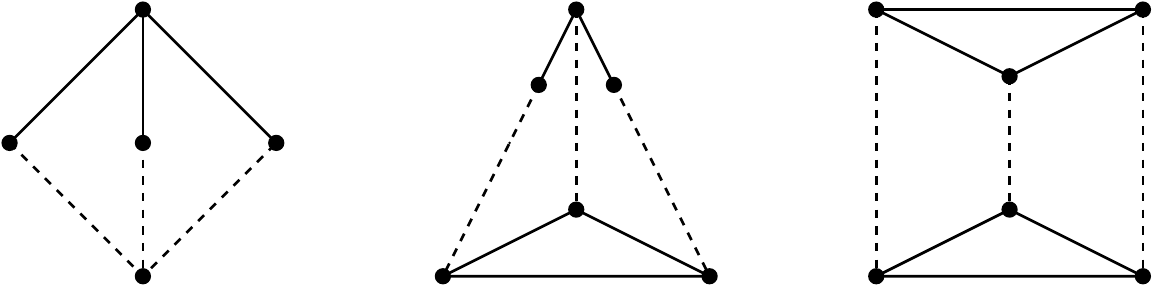}
\end{center}
\caption{Three-path-configurations: theta (left), pyramid (center), and prism (right). A full line represents an edge, and a dashed line represents a path that has at least one edge.} \label{fig:Truemper}
\end{figure}

A \emph{hole} in a graph $G$ is an induced cycle of $G$ of length at least four.
A \emph{wheel} is a graph that consists of a hole and an additional vertex that has at least three neighbors in the hole.
A \emph{broken wheel} (see \cref{fig:BrokenWheel}) is a wheel that consists of a hole $H$ and an additional vertex $v$ such that $v$ has at least three neighbors in $H$, and furthermore, the neighbors of $v$ in $V(H)$ induce a disconnected subgraph of $H$.

A \emph{prism} is any subdivision of $\overline{C_6}$ in which the two triangles remain unsubdivided; in particular, $\overline{C_6}$ is a prism. A \emph{pyramid} is any subdivision of the complete graph $K_4$ in which one triangle remains unsubdivided, and of the remaining three edges, at least two edges are subdivided at least once.
A \emph{theta} is any subdivision of the complete bipartite graph $K_{2,3}$; in particular, $K_{2,3}$ is a theta.
A \emph{3-path-configuration} (or \emph{3PC} for short) is any theta, pyramid, or prism. The three types of 3PC are represented in \cref{fig:Truemper}.

The \emph{short prism} is the graph $\overline{C_6}$, and a \emph{long prism} is any prism other than $\overline{C_6}$.
(Thus, a long prism is any prism on at least seven vertices.)

For an integer $n \geq 3$, a \emph{short $n$-prism} is a graph whose vertex set can be partitioned into two $n$-vertex cliques, say $A = \{a_1,\dots,a_n\}$ and $B = \{b_1,\dots,b_n\}$, such that for all $i,j \in \{1,\dots,n\}$, $a_i$ is adjacent to $b_j$ if and only if $i = j$. Note that $\overline{C_6}$ is a short $3$-prism, while for $n\ge 4$, no short $n$-prism is a prism.
A \emph{complete prism} is any graph that is a short $n$-prism for some integer $n \geq 3$.

We will need the following decomposition theorem for (3PC, wheel)-free graphs.

\begin{theorem}[Conforti et al.~\cite{conforti1997universally}] \label{thm-univ-sign}
If a graph $G$ is (3PC, wheel)-free, then either $G$ is a complete graph or a cycle, or $G$ admits a clique-cutset.
\end{theorem}

\begin{lemma} \label{lemma-GG2-Truemper-free-prelim}
$K_{2,3}$ is an induced minor of every theta, pyramid, long prism, or broken wheel.
\end{lemma}

\begin{proof}
First, we show that $K_{2,3}$ is an induced minor of every theta.
Let $H$ be a theta.
Let $a$ and $b$ be distinct, nonadjacent vertices of $H$, and let $P^1,P^2,P^3$ be distinct induced paths in $H$, each between $a$ and $b$, such that any two of $P^1,P^2,P^3$ have exactly two vertices (namely $a$ and $b$) in common.
Contracting in $H$ all but two edges of each path $P^i$ results in a graph isomorphic to $K_{2,3}$.

Next, we show that $K_{2,3}$ is an induced minor of every pyramid.
Let $H$ be a pyramid.
Let $a$ be a vertex of $H$, let $B = \{b_1,b_2,b_3\}$ be a 3-vertex clique in $H \setminus a$, and let $P^1$, $P^2$, and $P^3$ be induced paths in $H$ such that
\begin{itemize}
\item for each $i \in \{1,2,3\}$, the endpoints of $P^i$ are $a$ and $b_i$;
\item any two of the paths $P^1,P^2,P^3$ have exactly one vertex (namely $a$) in common.
\end{itemize}
Since $H$ is a pyramid, we know that at least two of $P^1,P^2,P^3$ have more than one edge; by symmetry, we may assume that $P^1$ and $P^2$ each have at least two edges.
Contracting in $H$ all but two edges of each of the paths $P^1$ and $P^2$, all but one edge of the path $P^3$, and the edge $b_1b_2$ results in a graph isomorphic to $K_{2,3}$.

Next, we show that $K_{2,3}$ is an induced minor of every long prism.
Let $H$ be a long prism.
Let $A = \{a_1,a_2,a_3\}$ and $B = \{b_1,b_2,b_3\}$ be disjoint 3-vertex cliques in $H$, and let $P^1$, $P^2$, and $P^3$ be induced paths in $H$ such that
\begin{itemize}
\item for each $i \in \{1,2,3\}$, the endpoints of $P^i$ are $a_i$ and $b_i$;
\item no two of the paths $P^1,P^2,P^3$ have any vertices in common.
\end{itemize}
Since $H$ is a long prism, we know that at least one of $P^1,P^2,P^3$ has more than one edge; by symmetry, we may assume that $P^1$ has more than one edge.
Contracting in $H$ all but two edges of the path $P^1$, all but one edges of each of the paths $P^2$ ad $P^3$, and the edges $a_1a_2$ and $b_1b_3$ results in a graph isomorphic to $K_{2,3}$.

Finally, we show that $K_{2,3}$ is an induced minor of every broken wheel.
Let $W$ be a broken wheel, consisting of a hole $H = h_0,h_1,\dots,h_{k-1},h_0$ (with $k \geq 4$, and indices in $\mathbb{Z}_k$) and an additional vertex $v$ that has at least three neighbors in $H$, and such that the neighbors of $v$ induce a disconnected subgraph of $H$.
By symmetry, we may assume that $v$ is nonadjacent to $h_0$ and adjacent to $h_1$.
Let the path $h_1,\dots,h_i$ be one component of $H[N_W(v)]$.
Contracting in $W$ all edges of the hole $H$ except for the four edges $h_0h_1, h_0h_{k-1}, h_{i}h_{i+1}, h_{i+1}h_{i+2}$ results in a graph isomorphic to $K_{2,3}.$
\end{proof}

\begin{corollary} \label{lemma-GG2-Truemper-free} Every graph in $\mathcal{G}_2$ is (theta, pyramid, long prism, broken wheel)-free.
\end{corollary}

\begin{proof}
Since $\mathcal{G}_2$ is hereditary (by \cref{prop-GC-hereditary}), it suffices to show that $\mathcal{G}_2$ contains no theta, no pyramid, no long prism, and no broken wheel.
By \cref{cor-K2k1}, $\mathcal{G}_2$ is a subclass of the class of $K_{2,3}$-induced-minor-free graphs.
Thus, it suffices to show that the class of $K_{2,3}$-induced-minor-free graphs contains no theta, no pyramid, no long prism, and no broken wheel, or equivalently, that $K_{2,3}$ is an induced minor of every theta, pyramid, long prism, or broken wheel.
This is exactly the statement of \Cref{lemma-GG2-Truemper-free-prelim}.
\end{proof}

\begin{lemma} \label{lemma-GG2-diamond-Truemper-free}
Let $G$ be a diamond-free graph that belongs to $\mathcal{G}_2$.
Then $G$ is (theta, pyramid, long prism, wheel)-free.
\end{lemma}

\begin{proof}
Clearly, every wheel either contains an induced diamond or is a broken wheel.
The result now follows from \cref{lemma-GG2-Truemper-free}.
\end{proof}

\begin{lemma} \label{lemma-C6-complement} Let $G$ be a (diamond, theta, pyramid, long prism, wheel)-free graph that contains an induced $\overline{C_6}$. Then either $G$ is a complete prism, or $G$ admits a clique-cutset.
\end{lemma}

\begin{proof}
Recall that $\overline{C_6}$ is a short 3-prism; fix a maximum integer $n \geq 3$ such that $G$ contains an induced short $n$-prism $H$. Set $V(H) = A \cup B$, where $A = \{a_1,\dots,a_n\}$ and $B = \{b_1,\dots,b_n\}$ are disjoint, $n$-vertex cliques, such that for all $i,j \in \{1,\dots,n\}$, $a_i$ is adjacent to $b_j$ in $H$ if and only if $i = j$. We may assume that $V(H) \subsetneqq V(G)$, for otherwise, $G$ is a complete prism, and we are done. We may further assume that $G$ is connected, for otherwise, $\emptyset$ is a clique-cutset of $G$, and again we are done.

\medskip
\noindent\textbf{Claim~1.} For all $v \in V(G) \setminus V(H)$, either $N_G(v) \cap V(H) = A$, or $N_G(v) \cap V(H) = B$, or there exists some $i \in \{1,\dots,n\}$ such that $N_G(v) \cap V(H) \subseteq \{a_i,b_i\}$.

\medskip
\noindent\emph{Proof of Claim~1.}
Fix $v \in V(G) \setminus V(H)$. We may assume that $|N_G(v) \cap V(H)| \geq 2$, for otherwise, the result is immediate. Next, if there exist distinct $i,j \in \{1,\dots,n\}$ such that $v$ is complete to $\{a_i,b_j\}$, then $G[v,a_i,a_j,b_i,b_j]$ is either a theta or a wheel, contrary to the fact that $G$ is (theta, wheel)-free. So, if $v$ has a neighbor both in $A$ and in $B$, then there exists some $i \in \{1,\dots,n\}$ such that $N_G(v) \cap V(H) = \{a_i,b_i\}$, and we are done. From now on, we assume that either $N_G(v) \cap V(H) \subseteq A$ or $N_G(v) \cap V(H) \subseteq B$; by symmetry, we may assume that $N_G(v) \cap V(H) \subseteq A$, and we deduce that $|N_G(v) \cap A| \geq 2$. Then $v$ is complete to $A$, for otherwise, we fix pairwise distinct $a_i,a_j,a_k \in A$ such that $v$ is adjacent to $a_i,a_j$ and nonadjacent to $a_k$, and we observe that $G[v,a_i,a_j,a_k]$ is a diamond, contrary to the fact that $G$ is diamond-free. It now follows that $N_G(v) \cap V(H) = A$, and we are done. This proves Claim~1. \hfill~$\blacklozenge$

\medskip
\noindent\textbf{Claim~2.} If there exists some $v \in V(G) \setminus V(H)$ such that $N_G(v) \cap V(H) = A$ (resp.\ such that $N_G(v) \cap V(H) = B$), then $A$ (resp.\ $B$) is a clique-cutset of $G$.

\medskip
\noindent\emph{Proof of Claim~2.}
By symmetry, we may assume that some $v \in V(G) \setminus V(H)$ satisfies $N_G(v) \cap V(H) = A$; we must show that $A$ is a clique-cutset of $G$. By construction, $A$ is a clique of $G$, and so it suffices to show that $A$ is a cutset of $G$ separating $v$ from $B$. Suppose otherwise. Then there exists an induced path $P$ in $G \setminus V(H)$ between $v$ and some vertex that has a neighbor in $B$. Let $Q = q_0,\dots,q_t$ (with $t \geq 0$) be minimum-length subpath of $P$ such that $q_0$ is complete to $A$ and $q_t$ has a neighbor in $B$; by Claim~1, $N_G(q_0)\cap V(H)=A$ and $q_0 \neq q_t$, i.e., $t \geq 1$. By symmetry, we may assume that $q_t$ is adjacent to $b_1$. By Claim~1, we have that either $N_G(q_t) \cap V(H) = B$, or $N_G(q_t) \cap V(H) = \{a_1,b_1\}$, or $N_G(q_t) \cap V(H) = \{b_1\}$.

Assume first that $N_G(q_t) \cap V(H) = B$. If $t = 1$, then $G[A \cup B \cup \{q_0,q_1\}]$ is a short $(n+1)$-prism, contrary to the maximality of $n$. So, $t \geq 2$. By the minimality of $Q$, all internal vertices of $Q$ are anticomplete to $B$. If the internal vertices of $Q$ are also anticomplete to $A$, then $G[\{a_1,a_2,b_1,b_2\} \cup V(Q)]$ is a long prism, contrary to the fact that $G$ is long-prism-free.
Hence, some internal vertex of $Q$ has a neighbor in $A$; let $i \in \{1,\dots,t-1\}$ be maximum with the property that $q_i$ has a neighbor in $A$. By the minimality of $Q$, and by Claim~1, we know that $q_i$ has a unique neighbor in $A$; fix $j \in \{1,\dots,n\}$ such that $a_j$ is the unique neighbor of $q_i$ in $A$, and fix any $k \in \{1,\dots,n\} \setminus \{j\}$. But now $G[a_j,a_k,b_j,b_k,q_i,q_{i+1},\dots,q_t]$ is a pyramid, contrary to the fact that $G$ is pyramid-free.

Assume next that $N_G(q_t) \cap V(H) = \{a_1,b_1\}$.
Let $q_i$ be the vertex of $Q$ with highest index such that $q_i$ is adjacent to $a_2$.
Then $q_i, \ldots, q_t,b_1,b_2,a_2,q_i$ is a hole, and $a_1$ has at least three neighbors (namely, $a_2,q_t,b_1$) in it, contrary to the fact that $G$ is wheel-free.

Assume finally that $N_G(q_t) \cap V(H) = \{b_1\}$.
Let $q_i$ (resp.\ $q_j$) be the vertex of $Q$ with highest index such that $q_i$ (resp.\ $q_j$) is adjacent to $a_2$ (resp.\ $a_1$).
If $j\geq i$ then $q_i, \ldots ,q_t,b_1,b_2,a_2,q_i$ is a hole and $a_1$ has at least three neighbors in it (namely, $a_2,q_j,b_1$), contrary to the fact that $G$ is wheel-free. So,  $i > j$. Then $q_j, \ldots ,q_t,b_1,a_1,q_j$ is a hole and $a_2$ has two nonadjacent neighbors in it (namely, $a_1,q_i$), and hence $G[q_j,\dots,,q_t,b_1,a_1,a_2]$ is a theta or a wheel, contrary to the fact that $G$ is (theta, wheel)-free. This proves Claim~2.~\hfill $\blacklozenge$

\medskip

In view of Claims~1 and~2, we assume from now on that for all $v \in V(G) \setminus V(H)$, there exists some $i \in \{1,\dots,n\}$ such that $N_G(v) \cap V(H) \subseteq \{a_i,b_i\}$.

\medskip
\noindent\textbf{Claim~3.} For all $i \in \{1,\dots,n\}$, if some vertex in $V(G) \setminus V(H)$ is complete to $\{a_i,b_i\}$, then $\{a_i,b_i\}$ is a clique-cutset of $G$.

\medskip
\noindent\emph{Proof of Claim~3.} 
By symmetry, we may assume that some vertex of $V(G) \setminus V(H)$ is complete to $\{a_1,b_1\}$; we must show that $\{a_1,b_1\}$ is a clique-cutset of $G$. 
Let $p_0$ be a vertex of $V(G) \setminus V(H)$ that is complete to $\{a_1,b_1\}$.
Since $a_1$ is adjacent to $b_1$, it suffices to show that $G \setminus \{a_1,b_1\}$ is disconnected. Suppose otherwise. 
Then, there exists an induced path in $G \setminus \{a_1,b_1\}$ from the vertex $p_0$ to a vertex in $V(H) \setminus \{a_1,b_1\}$. 
Consequently,
there exists an induced path $P = p_0,\dots,p_s$ (with $s \geq 0$) in $G \setminus (A \cup B)$ such that $p_s$ has a neighbor in $V(H) \setminus \{a_1,b_1\}$; we may assume that the path $P$ was chosen so that its length is minimum. 
By Claim~1, we have that $N_G(p_0) \cap V(H) = \{a_1,b_1\}$, and so $s \geq 1$. 
Furthermore, by the minimality of $P$, $\{p_0,\dots,p_{s-1}\}$ is anticomplete to $V(H) \setminus \{a_1,b_1\}$. 
By symmetry, we may assume that $a_2 \in N_G(p_s) \cap V(H) \subseteq \{a_2,b_2\}$. Let $i$ be the largest index in $\{0,\dots,s\}$ such that $p_i$ is adjacent to $b_1$ (such an $i$ exists because $p_0$ is adjacent to $b_1$). Now $p_i,\dots,p_s,a_2,a_3,b_3,b_1,p_i$ is a hole, and $a_1$ has at least three neighbors (namely, $a_2,a_3,b_1$) in it, contrary to the fact that $G$ is wheel-free. This proves Claim~3. \hfill $\blacklozenge$

\medskip
In view of Claim~3, we may now assume that no vertex in $V(G) \setminus V(H)$ has more than one neighbor in $V(H)$. Let $N_A$ be the set of all vertices in $V(G) \setminus V(H)$ that have a neighbor in $A$, and let $N_B$ be the set of all vertices in $V(G) \setminus V(H)$ that have a neighbor in $B$. Then $N_A \cap N_B = \emptyset$. Since $V(H) \subsetneqq V(G)$ and $G$ is connected, we have that $N_A \cup N_B \neq \emptyset$. If $N_A = \emptyset$, then $B$ is a clique-cutset of $G$, and if $N_B = \emptyset$, then $A$ is a clique-cutset of $G$.
So, we may assume that $N_A$ and $N_B$ are both nonempty.
Furthermore, we may assume that there is an induced path in $G \setminus V(H)$ between $N_A$ and $N_B$, for otherwise, both $A$ and $B$ are clique-cutsets of $G$, and we are done. Let $P = p_0,\dots,p_s$ (with $s \geq 0$) be a minimum-length path in $G \setminus V(H)$ such that $p_0 \in N_A$ and $p_s \in N_B$; since $N_A \cap N_B = \emptyset$, we see that $s \geq 1$. Furthermore, the minimality of $P$ implies that the interior of $P$ is anticomplete to $V(H)$.

Since no vertex of $V(G) \setminus V(H)$ has more than one neighbor in $V(H)$, we may assume by symmetry that $N_G(p_0) \cap V(H) = \{a_1\}$, and that either $N_G(p_s) \cap V(H) = \{b_1\}$ or $N_G(p_s) \cap V(H) = \{b_2\}$. But if $N_G(p_s) \cap V(H) = \{b_2\}$, then $G[a_1,a_2,b_1,b_2,p_0,\dots,p_s]$ is a theta, contrary to the fact that $G$ is theta-free. So, $N_G(p_s) \cap V(H) = \{b_1\}$. We now have that $V(P)$ is anticomplete to $V(H) \setminus \{a_1,b_1\}$.

Our goal is to show that $\{a_1,b_1\}$ is a clique-cutset of $G$. Suppose otherwise; then $G \setminus \{a_1,b_1\}$ is connected.
Then, there exists an induced path in $G \setminus \{a_1,b_1\}$ from a vertex in $P$ to a vertex in $V(H) \setminus \{a_1,b_1\}$. 
Since $V(P)$ is anticomplete to $V(H) \setminus \{a_1,b_1\}$, any such path has length at least two.
Deleting the endpoints of any such path, we obtain an induced path $Q = q_0,\dots,q_t$ (with $t \geq 0$) in $G \setminus (V(H) \cup V(P))$ such that $q_0$ has a neighbor in $V(P)$, and $q_t$ has a neighbor in $V(H) \setminus \{a_1,b_1\}$; we may assume that $Q$ is a minimum-length path with this property, so that $q_0$ is the only vertex of $Q$ with a neighbor in $V(P)$, and $q_t$ is the only vertex of $Q$ with a neighbor in $V(H) \setminus \{a_1,b_1\}$. By symmetry, we may further assume that $q_t$ is adjacent to $a_2$; then $N_G(q_t) \cap V(H) = \{a_2\}$. Let $i$ be the largest index in $\{0,\dots,s\}$  such that $q_0$ is adjacent to $p_i$. Then $p_i,\dots,p_s,b_1,b_3,a_3,a_2,q_t,\dots,q_0,p_i$ is a hole in $G$, and $b_2$ has at least three neighbors (namely, $a_2,b_1,b_3$) in it, contrary to the fact that $G$ is wheel-free.
This completes the proof.
\end{proof}

\begin{lemma} \label{lemma-GG2-diamond-decomp} Let $G$ be a (diamond, theta, pyramid, long prism, wheel)-free graph. Then either $G$ is a complete prism, a cycle, or a complete graph, or $G$ admits a clique-cutset.
\end{lemma}

\begin{proof}
If $G$ contains an induced $\overline{C_6}$, then the result follows from \cref{lemma-C6-complement}. Otherwise, we have that $G$ is (3PC, wheel)-free, and the result follows from \cref{thm-univ-sign}.
\end{proof}

\begin{theorem} \label{thm-GG2-diamond-decomp}
Let $G$ be a diamond-free graph that belongs to $\mathcal{G}_2$. Then either $G$ is a complete prism, a cycle, or a complete graph, or $G$ admits a clique-cutset.
\end{theorem}

\begin{proof}
This follows immediately from \cref{lemma-GG2-diamond-Truemper-free,lemma-GG2-diamond-decomp}.
\end{proof}

\subsection{Algorithmic considerations}

Clearly, complete prisms, cycles, and complete graphs are diamond-free and belong to $\mathcal{G}_2$, and furthermore, they can all be recognized in polynomial time.
So, using \cref{cor-G2-clique-gluing} and \cref{thm-GG2-diamond-decomp}, we show that diamond-free graphs in $\mathcal{G}_2$ can be recognized in polynomial time.
In order to derive the result, we decompose a graph by means of clique-cutsets.
This common algorithmic tool, first proposed by Tarjan  (\cite{tarjan1985decomposition}), applies to any $n$-vertex graph $G$ and produces a family $\mathcal{H}$ of $\mathcal{O}(n)$ induced subgraphs of $G$ that do not have any clique-cutsets and such that $G$ can be obtained by an iterative application of gluing graphs from $\mathcal{H}$ along cliques.
The original algorithm proposed by Tarjan runs in time $\mathcal{O}(n(n+m))$, where $m$ denotes the number of edges of $G$. 
A more efficient approach for decomposing a graph along clique-cutsets was suggested by Coudert and Ducoffe \cite{MR3780116}.
They improved the time complexity to $\mathcal{O}(n^{\omega}\log n),$ where $\omega<2.3728596$ is the matrix multiplication exponent (see~\cite{MR4262465}).

\begin{proposition}\label{prop:recognition-diamond-free}
There exists an algorithm running in time $\mathcal{O}(n^{\omega}\log n)$ that correctly determines if an input $n$-vertex graph $G$ is a diamond-free graph in $\mathcal{G}_2$.
\end{proposition}

\begin{proof}
Given a graph $G$ with $n$ vertices, testing if $G$ is diamond-free can be done in time $\mathcal{O}(n^{\omega})$ (see~\cite{MR3451135}).
Assuming $G$ is diamond-free, we compute the connected components of $G$ and run the
algorithm by Coudert and Ducoffe~\cite{MR3780116} on each nontrivial component of $G$.
This can be done in time $\mathcal{O}(n^{\omega}\log n)$.
The algorithm produces a family $\mathcal{H}$ of $\mathcal{O}(n)$ induced subgraphs of $G$ that do not have any clique-cutsets and such that $G$ can be obtained by an iterative application of gluing graphs from $\mathcal{H}$ along cliques.
We then check, for each graph $H\in \mathcal{H}$, whether $H$ is a complete prism, cycle, or a complete graph.
If this is the case, the algorithm determines that $G$ belongs to $\mathcal{G}_2$, and otherwise, it determines that $G$ does not belong to $\mathcal{G}_2$.
The correctness follows from \cref{cor-G2-clique-gluing} and \cref{thm-GG2-diamond-decomp}.

To complete the proof, we show that testing whether a given $H\in \mathcal{H}$ satisfies one of the desired properties can be done in time $\mathcal{O}(n+m)$, where $m$ denotes the number of edges of $G$.
Since $H$ is connected, testing if it is a cycle or a complete graph can be done in linear time simply by checking if all the vertex degrees are equal to $2$ or to $|V(H)|-1$, respectively.
If none of these cases occurs, we can assume that $n = 2k$ for some $k\ge 3$ and that every vertex in $H$ has degree exactly $k$, since otherwise, we can infer that $H$ is not a complete prism.
We choose an arbitrary vertex $v\in V(H)$ and compute the components of the graph $H[N_H(v)]$.
If $H$ is a short $k$-prism, then $H[N_H(v)]$ has exactly two components, say $C$ and $D$, such that $C$ is isomorphic to a complete graph $K_{k-1}$ and $D$ is a trivial component.
Set $A = C\cup \{v\}$ and $B = V(G)\setminus A$.
Since we already checked that all the vertices are of degree $k+1$, it remains to verify if $B$ is a clique of cardinality $k$.
If this is the case, then $H$ is a complete prism, otherwise it is not.
Each of the above constantly many steps can be carried out in linear time.
\end{proof}

Moreover, it is clear that the \textsc{Maximum Weight Clique}, \textsc{Maximum Weight Stable Set}, and \textsc{Vertex Coloring} can be solved in polynomial time for complete prisms, cycles, and complete graphs.
Thus, \cref{thm-GG2-diamond-decomp} and the algorithm by Coudert and Ducoffe \cite{MR3780116} allow us to solve these three optimization problems in polynomial time for diamond-free graphs in $\mathcal{G}_2$.
A more precise time complexity analysis is provided by the following.

\begin{theorem}\label{diamond-free-GG-2-algo}
When restricted to the class of diamond-free graphs in $\mathcal{G}_2$ with $n$ vertices and $m$ edges, the \textsc{Maximum Weight Clique} and \textsc{Vertex Coloring} problems can be solved in $\mathcal{O}(n^{\omega}\log n)$ time and the \textsc{Maximum Weight Stable Set} problem in $\mathcal{O}(n^2(n+m))$ time.
\end{theorem}

\begin{proof}
Let $G$ be a diamond-free graph with $n$ vertices and $m$ edges that belongs to $\mathcal{G}_2$.
For the \textsc{Maximum Weight Clique} and \textsc{Vertex Coloring} problems, the approach is as follows.
We compute the connected components of $G$ and run the algorithm by Coudert and Ducoffe~\cite{MR3780116} on each component of $G$.
This can be done in time $\mathcal{O}(n^{\omega}\log n)$.
We obtain a family $\mathcal{H}$ of $\mathcal{O}(n)$ induced subgraphs of $G$ that do not have any clique-cutsets and such that $G$ can be obtained by an iterative application of gluing graphs from $\mathcal{H}$ along cliques.
We now iterate over all $H\in \mathcal{H}$ and solve the \textsc{Maximum Weight Clique} and \textsc{Vertex Coloring} problems on $H$ in linear time.
By \cref{thm-GG2-diamond-decomp}, each graph $H\in \mathcal{H}$ is a complete prism, cycle, or a complete graph; as explained in the proof of \cref{prop:recognition-diamond-free}, which of these cases occurs can be determined in linear time in the size of $H$.
If $H$ is a complete graph, then its chromatic number is $|V(H)|$ and its vertex set solves the \textsc{Maximum Weight Clique} problem.
Otherwise, if $H$ is a cycle with at least four vertices, then its chromatic number is either $2$ or $3$, depending on whether $|V(H)|$ is even or odd, respectively, and to solve the \textsc{Maximum Weight Clique} problem, we only need to examine its edges.
Finally, if $H$ is a complete prism, say with $|V(H)| = 2k$ for some $k\ge 3$, then we can identify in linear time the two cliques $A$ and $B$, each of size $k$, that partition $V(H)$.
Then, the chromatic number of $H$ is $k$, and to solve the \textsc{Maximum Weight Clique} problem, we only need to examine the two cliques $A$ and $B$ and the $k$ edges connecting them.
In all these cases, the \textsc{Maximum Weight Clique} and \textsc{Vertex Coloring} problems can be solved in linear time for graphs in $\mathcal{H}$.
Since each clique of $G$ is fully contained in one of the graphs in $\mathcal{H}$, this provides an efficient solution to the \textsc{Maximum Weight Clique} problem on $G$.
Similarly, the chromatic number of $G$ is the maximum chromatic number of the graphs in~$\mathcal{H}$.

For the \textsc{Maximum Weight Stable Set} problem, the approach is similar, except that in each decomposition step decomposes $G$ along a cut-partition $(A,B,C)$ of $G$ such that $C$ is a clique and the subgraph of $G$ induced by $A\cup C$ belongs to $\mathcal{H}$.
For each vertex $v\in C$, we determine a maximum-weight stable set of the subgraph of $G$ induced by $A\setminus N_G(v)$, a maximum-weight stable set of the subgraph of $G$ induced by $A$, redefine the weights on $C$, and solve the problem recursively on the graph $G-A$. (We refer to~\cite{tarjan1985decomposition} for details; see also~\cite[Section 8.1]{MR3948128}.)
Thus, we solve $\mathcal{O}(|V(G)|)$ subproblems for each graph $H\in \mathcal{H}$ for a total of $\mathcal{O}(|V(G)|^2)$ subproblems.
Each of these subproblems can be solved in linear time.
If $H$ is a complete graph, then a heaviest vertex forms a maximum-weight stable set.
If $H$ is a cycle, then we can use the fact that cycles have bounded treewidth and apply the results from~\cite{MR1105479,MR1417901}.
Finally, assume that $H$ is a complete prism.
Then, for each vertex $v$ a maximum-weight stable set $S_v$ containing $v$ can be computed in $\mathcal{O}(|V(H)|)$ time: indeed, $S_v$ is of the form $S_v = \{v,z^v\}$ where $z^v$ is a vertex of maximum weight among the nonneighbors of $v$ in $H$. 
The heaviest among the sets $S_v$ forms a maximum-weight stable set.
The complexity of this approach is $\mathcal{O}(|V(H)|^2)$, which in this case is $\mathcal{O}(|V(H)| + |E(H)|)$.
\end{proof}

In conclusion, let us put the results of \cref{diamond-free-GG-2-algo} in perspective by comparing them with the known complexities of the three problems in the larger classes of diamond-free graphs and graphs in $\mathcal{G}_2$.
First, the \textsc{Vertex Coloring} problem is \NP-hard for diamond-free graphs~\cite{MR1905637}, as well as for graphs in $\mathcal{G}_2$, since it is already hard  for the subclass of circular-arc graphs~\cite{CircArcColNPComp}.
The situation is somewhat different for the  \textsc{Maximum Weight Stable Set} problem, which is \NP-hard (even in the unweighted case) in the class of diamond-free graphs, as can be seen using Poljak's reduction~\cite{poljak}, but solvable in $\mathcal{O}(n^6)$ time for $n$-vertex graphs in $\mathcal{G}_2$ (see~\cref{MWIS-GG_k}).

Finally, while the \textsc{Maximum Weight Clique} problem is known to be solvable in polynomial time both for diamond-free graphs as well as for graphs in $\mathcal{G}_2$, the running time of the algorithm given by \cref{diamond-free-GG-2-algo} improves on both time complexities.
By \cref{MWC-G22}, the problem can be solved in $\mathcal{O}(n^4)$ time for $n$-vertex graphs in $\mathcal{G}_2$.
For the class of diamond-free graphs, observe that every edge in such a graph is contained in a unique maximal clique.
Thus, a diamond-free graph with $n$ vertices and $m$ edges has $\mathcal{O}(n+m)$ maximal cliques, and the \textsc{Maximum Weight Clique} problem can be solved in polynomial time by enumerating all maximal cliques and returning one of maximum weight.
Using, for example, the maximal clique enumeration algorithm due to Makino and Uno~\cite{MR2159537}, this would result in an overall running time of $\mathcal{O}(n^{2.373}(n+m))$ on diamond-free graphs with $n$ vertices and~$m$~edges.

\subsection*{Acknowledgement}

The authors are grateful to the anonymous reviewers for their careful reading of the manuscript and valuable suggestions.

\bibliographystyle{abbrv}

\end{document}